\documentclass[aos]{imsart}

\RequirePackage{amsthm,amsmath,amsfonts,amssymb}
\RequirePackage[authoryear]{natbib}%% uncomment this for author-year citations
\usepackage{comment}

\usepackage[all,cmtip]{xy}
\usepackage{mathrsfs}
\usepackage{graphicx}
\usepackage{lscape}
\usepackage[nodayofweek,level]{datetime}
\usepackage[vlined,boxed,ruled]{algorithm2e}
\usepackage{algorithmicx}
\usepackage{algpseudocode}
\usepackage{hyperref}
\usepackage{cleveref}
\usepackage{enumitem}
\usepackage{caption}
\usepackage[subrefformat=parens,labelformat=parens]{subcaption}

\usepackage{booktabs}

\usepackage[x11names]{xcolor}

\startlocaldefs

\makeatletter
\newcommand *{\eqv}{\mathrel{\rlap{\raisebox{0.3ex}{$\m@th\cdot$ }}\raisebox {-0.3ex}{$\m@th\cdot$}}=}
\makeatother

\theoremstyle{plain}

\newtheorem{theorem}{Theorem}[section]

\newtheorem{proposition}[theorem]{Proposition}

\theoremstyle{remark}
\newtheorem{example}[theorem]{Example}
\newtheorem{definition}[theorem]{Definition}

\newtheorem{remark}[theorem]{Remark}
\numberwithin{equation}{section}
\usepackage{xcolor}

\newcommand{\GL}{\mathrm{GL}}

\newcommand{\SO}{\mathrm{SO}}

\newcommand{\SU}{\mathrm{SU}}

\newcommand{\tr}{\operatorname{tr}}
\newcommand{\Ad}{\operatorname{Ad}}
\newcommand{\ad}{\operatorname{ad}}
\newcommand{\End}{\operatorname{End}}
\newcommand{\trace}{\tr}

\newcommand{\dd}{\mathrm{d}}
\newcommand{\ee}{\mathrm{e}}
\newcommand{\ii}{\mathrm{i}}

\newcommand{\cb}{\mathbf c}
\newcommand{\phib}{\boldsymbol{\phi}}
\newcommand{\Eb}{\mathbf E}

\newcommand{\C}{\mathbb C}
\newcommand{\N}{\mathbb N}
\newcommand{\R}{\mathbb R}
\newcommand{\Z}{\mathbb Z}

\newcommand{\gf}{\mathfrak g}

\newcommand{\kf}{\mathfrak k}
\newcommand{\tf}{\mathfrak t}

\usepackage{mathtools}
\newcommand{\defeq}{\vcentcolon=}

\newcommand{\range}{\mathrm{range}}

\newcommand{\diag}{\operatorname{diag}}

\endlocaldefs

\begin{document}

\begin{frontmatter}
\title{Optimal Designs for Regression on Lie Groups}
%\title{A sample article title with some additional note\thanksref{t1}}
\runtitle{Optimal Designs for Regression on Lie Groups}
%\thankstext{T1}{A sample additional note to the title.}

\begin{aug}
%%%%%%%%%%%%%%%%%%%%%%%%%%%%%%%%%%%%%%%%%%%%%%%
%% Only one address is permitted per author. %%
%% Only division, organization and e-mail is %%
%% included in the address.                  %%
%% Additional information can be included in %%
\author[]{\fnms{Somnath}~\snm{Chakraborty}\ead[label=e1]{somnath.chakraborty@helsinki.fi}},
\author[]{\fnms{Holger}~\snm{Dette}\ead[label=e2]{holger.dette@rub.de}}, and
\author[]{\fnms{Martin}~\snm{Kroll}\ead[label=e3]{martin.kroll@uni-bayreuth.de}}
% look at line 3084 of imsart.sty to bring back 'By'

%%%%%%%%%%%%%%%%%%%%%%%%%%%%%%%%%%%%%%%%%%%%%%
%% Addresses                                %%
%%%%%%%%%%%%%%%%%%%%%%%%%%%%%%%%%%%%%%%%%%%%%%
\address[]{Department of Mathematics and Statistics, 
University of Helsinki,  
FI-00560 Helsinki, Finland\printead[presep={,\ }]{e1}}
\address[]{Fakult\"{a}t f\"{u}r Mathematik, Ruhr-Universit\"{a}t Bochum, D-44801 Bochum, Germany\printead[presep={,\ }]{e2}}
\address[]{Fakult\"{a}t f\"{u}r Mathematik, Physik und Informatik, Universit\"{a}t Bayreuth, D-95440 Bayreuth, Germany\printead[presep={,\ }]{e3}}
%\address[B]{Department, University or Company Name\printead[presep={,\ }]{e2,e3}}
\end{aug}

\begin{abstract}

We consider a linear regression model with complex-valued response and predictors from a compact and connected Lie group.
The regression model is formulated in terms of eigenfunctions of the Laplace-Beltrami operator on the Lie group.
We show that the normalized Haar measure is an \emph{approximate} optimal design with respect to all Kiefer's $\Phi_p$-criteria.
Inspired by the concept of $t$-designs in the field of algebraic combinatorics, we then consider so-called $\lambda$-designs in order to construct \emph{exact} $\Phi_p$-optimal designs for fixed sample sizes in the considered regression problem.
In particular, we explicitly construct $\Phi_p$-optimal designs for regression models with predictors in the Lie groups $\SU(2)$ and $\SO(3)$, the groups of $2\times 2$ unitary matrices and $3\times 3$ orthogonal matrices with determinant equal to $1$, respectively. 
We also discuss the advantages of the derived theoretical results in a concrete biological application.
\end{abstract}



\end{frontmatter}
%%%%%%%%%%%%%%%%%%%%%%%%%%%%%%%%%%%%%%%%%%%%%%
%% Please use \tableofcontents for articles %%
%% with 50 pages and more                   %%
%%%%%%%%%%%%%%%%%%%%%%%%%%%%%%%%%%%%%%%%%%%%%%
%\tableofcontents

%%%%%%%%%%%%%%%%%%%%%%%%%%%%%%%%%%%%%%%%%%%%%%
%%%% Main text entry area:

\section{Introduction}\label{sec:intro} The statistical analysis of non-Euclidean data has been an active area of research for several decades now. 
Contributions include density estimation on manifolds \citep{Hendriks1990,Kerkyacharian2011,Hielscher2013}, spherical regression \citep{Chang1986,Downs2003} where both covariates and responses are spherical, testing \citep{Gine1975,Lacour2014}, and the estimation of extrinsic resp. intrinsic means \citep{Bhattacharya2003,Bhattacharya2005}, just to mention a few.
Non-parametric regression models with predictors lying on a manifold and real-valued responses are of particular importance since they naturally arise in a variety of applications in statistics and other empirical studies \citep{Lee1996,Lin2019,Jeon2022}.
For the special case of spherical predictors non-parametric methods like kernel type estimators \citep{DiMarzio2009,DiMarzio2014}, spline estimators \citep{Wahba1981,Alfeld1996}, and series estimators \citep{Narcowich2006,Monnier2011,Lin2019a,Abrial2008} have been intensively studied.

In this paper, we consider the problem of designing experiments for a regression problem where the covariate lies in a manifold other than a sphere (or a product of spheres). Our work is motivated by a  model for a molecular biological scoring function introduced in \citet{Padhorny2016}.
Scoring functions are meant to predict the strength of a specific receptor-ligand binding and are an important building block in docking methods where the interest is to determine interaction energies for large samples of potential receptor-ligand orientations.
Docking methods, in turn, are of fundamental importance in concrete applications like structure-based drug-design \citep{Kitchen2004,Guedes2018}.
The idea of \citet{Padhorny2016} is to consider a scoring function defined on the product manifold $\mathrm{SO}(3)/ \mathbb S^1 \times \mathrm{SO}(3) \simeq \mathbb S^2 \times \mathrm{SO}(3)$, where 
$\mathbb S^p \subset \R^{p+1}$ denotes the $p$-dimensional sphere and 
$ \mathrm{SO}(p)$ the group of $p \times p$ orthonormal matrices; the symbol $\simeq $
means that there exists an isomorphism between the manifolds.
More precisely, in their model the $\SO(3)$ part represents the space of the rotating ligand, whereas the $\SO(3)/ \mathbb S^1 \simeq \mathbb S^2$ part is the space spanned by the two Euler angles that define the orientation of the vector from the center of the fixed receptor toward the center of the ligand.
In contrast to alternative methods available, a major advantage of this modelling lies in its computational efficiency by means of the availability of fast Fourier transform (FFT) methods on the rotation group, as presented in \citet{Kostelec2008}.
The scoring function of \citet{Padhorny2016} is defined as a generalized Fourier series with a limited number of basis functions that are defined as tensor products of (slightly modified) Wigner D-functions.
\citet{Padhorny2016} also notice the fact that the equally spaced sampling of Euler angles considered by them leads to an inherently non-uniform distribution on the sampling domain and that this irregular sampling might distort subsequent analyses:

\begin{quotation}
\normalsize{This \normalfont{[scil. the non-uniform sampling]} becomes a significant problem if one seeks to obtain statistical information about the energy landscape of protein interaction [...]}
\end{quotation}

\citet{Padhorny2016} propose an ad hoc thinning procedure tailored to their specific problem to address the non-uniformity of their original sampling.
However, their remark directly motivates the question how (in some sense) optimal designs on the sampling domain should be constructed.

Besides the application from biology described above the construction of optimal designs on manifolds might equally be of interest in other scientific areas. As a further example let us mention the field of materials science where bases consisting of (symmetrized) basis functions have recently been introduced for the representation of various functions of interest \citep{Mason2008,Patala2012,Mason2019}.
In particular, the manifold $\SO(3)$ (the rotation group in dimension three) arises naturally in various applications from engineering and crystallography \citep{Chirikjian2021,Schaeben2007,Hielscher2013,Kovacs2003}.

The question of designing experiments for these and other applications of regression on manifolds is by no means trivial, and answers have only been given in special cases, which are from a practical point of view not completely satisfactory.  
For the case when the underlying manifold is a sphere, optimal designs for finite series estimators in terms of the so-called (hyper-)spherical harmonics are determined in \citet{Dette2005} and \citet{Dette2019} with respect to a broad class of optimality criteria.
Recently, \citet{LiCastillo2024} determined $D$-optimal designs for regression models on manifolds.
However, these optimal designs are \emph{approximate} designs in the sense of  \citet{Kiefer1974}, which means that they are probability measures with finite support and cannot directly be implemented for a fixed sample size. 
More precisely, a standard argument from optimal design theory using Carath{\'e}odory's theorem and the results in the cited papers show that for the regression models defined by a series of (hyper-)spherical harmonics there always exists an optimal approximate  design with 
finite support, say $\xi=\{ x_1 , \ldots, x_m\}$, and corresponding weights 
$\omega_1, \ldots, \omega_m >0$ such that $\sum_{i=1}^m\omega_i=1$. However, in most applications $m$ is a large integer and 
for a fixed sample size, say  $n$, the quantities $n \omega_i$ are usually  not necessarily integers. In fact, in most cases they have to be rounded to integers, say $n_i$, such that $\sum_{i=1}^m n_i = n$, and $n_i$ observations are taken at each $x_i$ ($i=1,\ldots,m$). 
As a consequence, the resulting design will not necessarily remain optimal in the class of all 
{\it exact designs for sample size $n$}, that is the class of all probability measures with finite support and masses which are multiples of $1/n$.

The present paper has two objectives. First, we take a more general point of view on the problem of  experimental design for regression models with predictors from a manifold and a complex-valued response, say $Y$. More precisely, we consider the situation, where 
the predictor is defined on a compact and connected Lie group $G$ and the regression function is given by a finite linear combination in terms of a well-defined set of orthonormal basis functions on the space $L^2(G)=L^2_{\mu_G}(G)$ of square-integrable complex-valued functions with respect to the probability Haar measure $\mu_{G}$ on $G$.
We show that the measure $\mu_G$ is in fact an optimal approximate design for estimating the coefficients by the least squares approach.
Secondly, and from a practical point of view more importantly, we consider the problem of constructing implementable \emph{exact} designs for fixed sample sizes from the Haar probability measure.
In particular, we show that there is an efficient way of explicitly constructing several implementable designs for selected finite sample sizes.
Our approach is based on the notion of %so-called 
$\lambda$-designs, which are generalizations of spherical $t$-designs (see, for instance, \citet{BANNAI20091392}, for a survey on spherical designs) to the case of compact Riemannian manifolds.
Roughly speaking, $\lambda$-designs define quadrature formulas for %the basis 
functions in $L^2_{\mu_{G}}(G)$ up to a given resolution level.
So far, apart from theoretical contributions (see, for example, \citet{Bondarenko2013, Bannai2022, Gariboldi2021} and the references therein), $\lambda$-designs have been used for applications in quantum computing (see, for example, \citet{Hastings2009, Sen2018}, and the references therein, for results concerning designs on unitary groups).

To the best of our knowledge, approaching statistical optimal design criteria via $\lambda$-designs on Lie groups has not been considered yet in the statistics literature.

The rest of this paper is organized as follows.
In Section~\ref{sec:regression} we introduce basic notions, the regression problem and briefly revisit the the basic concepts of optimal design theory.
In Section~\ref{sec:opt2} we prove optimality of the normalized Haar measure on the Lie group with respect to 
many optimality criteria including Kiefer's $\Phi_p$-criteria, see 
\citet{Kiefer1974}.
Section~\ref{sec:lambda} is devoted to the concept of $\lambda$-designs and the construction of exact optimal designs by means of $\lambda$-designs.
In Section~\ref{sec:examples} we consider concrete examples. In particular, we give examples of optimal designs on the special unitary group $\SU(2)$ and the special orthogonal group $\SO(3)$.
Moreover we also construct exact optimal designs for the regression %model 
model for a molecular biological scoring function 
considered in \citet{Padhorny2016}.
For this and several other examples we demonstrate the advantages of our approach, which yields a substantial reduction of the sample sizes (compared to alternative design strategies) without losing statistical accuracy.
Finally, in Appendix~\ref{app:lie} we gather notions and results on Lie groups that are essential for the understanding of the main part of the paper and provide more  details for some of the proofs of our results.

\section{Approximate designs for linear regression on Lie groups}
\label{sec:regression}

Let $G$ be a compact and connected Lie group and $\mu_G$ %be 
the uniquely defined (left- and right-) invariant probability Haar measure on $G$.

Since every compact and connected Lie group $G$ is also a compact Riemannian manifold, an orthonormal basis of $L^2(G)=L^2_{\mu_G}(G)$ is given by the eigenfunctions of the negative Laplace-Beltrami operator $-\Delta_G=-\mathrm{div}\circ\nabla$, which is an analogue of the usual Euclidean (negative) Laplace operator given by
\begin{equation*}
    -\Delta=
    -\sum_{i=1}^n \frac{\partial^2}{\partial x_i^2}    
\end{equation*}
on $\R^n$.
For a precise definition of the Laplace-Beltrami operator on a compact Riemannian manifold in terms of local coordinates we refer to \citet{Gine1975}, p. 103.
Essential for the present paper is the decomposition 
\begin{equation*}\label{eq:pw1}
    L^2(G)
    \cong \bigoplus_{j=0}^\infty \, E_{\lambda_j}
\end{equation*}
in terms of the finite-dimensional eigenspaces $E_{\lambda_j}$ corresponding to the increasingly ordered eigenvalues $0=\lambda_0<\lambda_1<\lambda_2 < \cdots \to \infty$ of $-\Delta_{G}$.
This sequence of eigenvalues is commonly called the spectrum of the Riemannian manifold, see \citet{Berger1971}.
Let $s_j \defeq \dim (E_{\lambda_j})< \infty$ denote the dimension of $E_{\lambda_j}$, equivalently, 
%which is 
the multiplicity of the eigenvalue $\lambda_j$.
Note that $s_0=1$ holds for any compact and connected Riemannian manifold (see \citet{Berger1971}, Remarque A.I.5).
By fixing an orthonormal basis $\{ \phi_{j,k}, k = 1,2,\dots,s_j \}$ for %any
each eigenspace $E_{\lambda_j}$, we obtain an orthonormal basis
\begin{equation*}
    \{ \phi_{j,k} \mid j = 0,\ldots,d, k = 1,2,\dots,s_j \}
\end{equation*}
for the direct sum $\bigoplus_{j=0}^d E_{\lambda_j}$, which is a finite-dimensional vector space of dimension $D \defeq s_0+\ldots+s_d$, containing ''band-limited'' functions up to a certain resolution level.

Let $g$ denote a $G$-valued covariate.
We assume that at potentially any $g \in G$ it is possible to observe a complex-valued response $Y$.
More precisely, for the complex-valued response we assume the linear regression model
\begin{equation}\label{eq:regress01}
    \Eb [Y|g] = \phib^\top(g) \boldsymbol{c}
\end{equation}
where
\begin{equation}\label{eq:def:phib}
    \phib^\top = (\phi_{0,1}, \phi_{1,1},\ldots,\phi_{0,s_1}, \ldots. \phi_{d,1},\ldots , \phi_{d,s_d})~,
\end{equation}
is the vector of linearly independent, complex-valued regression functions and $$
\boldsymbol{c} =  (c_{0,1}, c_{1,1},\ldots,c_{0,s_1}, \ldots. c_{d,1},\ldots , c_{d,s_d})^\top \in \C^D
$$
the unknown parameter vector (we refer to \citet{Miller1973} for a detailed discussion of  
complex-valued linear models).
We further assume that different observations are independent with equal variance, say $\sigma^2 >0$. 
Let $\mathscr P_{G}$ be the set of all Borel probability measures on $G$.
For any $\mu \in \mathscr P_{G}$, we define the \emph{information matrix}
\begin{equation}\label{eq:info}
    M(\mu) \defeq \int_{G} \phib
    (g) \phib^\ast(g)\dd \mu(g)
\end{equation}
(here and throughout this paper, the superscript $^\ast$ denotes complex conjugation and transposition of a matrix, that is $\phib^\ast(g) = \overline{\phib}(g)^\top$),
and denote with
\begin{equation} \label{det2}
    \mathscr M_{G}(\phib) \defeq
    \{M(\mu) : \mu\in 
    \mathscr P_{G} \}
\end{equation} 
the set of all information matrices for probability measures $\mu$.
Note that the information matrix $M(\mu)$ is a non-negative definite $D\times D$ matrix (which is equivalent to $M(\mu)$ being Hermitian with non-negative eigenvalues only). 
Moreover, $\mathscr M_{G}(\phib)$ is the image of the convex subset 
$\mathscr P_{G} $ under the linear map $\mu \mapsto M(\mu)$ and therefore also convex.

Let $s\leq D$, and consider a matrix $K \in \C^{D \times s}$ 
having full (column) rank $s$.
Assume that an experimenter takes $n_i$ independent observations at each design point $g_i$ from a finite set $\{ g_1,\ldots,g_m \}$.
This experimental design corresponds to the probability measure $\mu_n \defeq \sum_{i=1}^m \frac{n_i}{n}\,\delta_{g_i}$ with masses $\frac{n_1}{n}, \ldots , \frac{n_m}{n}$ at the points. ${g_1}, \ldots , g_m  \in G$, respectively, where $n=\sum_{i=1}^m n_i$ is the total number of measurements.
If the {\it range inclusion}  condition
\begin{equation}\label{eq:cond:range:inclusion}
    \range(K) \subseteq \range(M(\mu_n))
\end{equation}
holds, it is well known that the covariance matrix of the Gauss-Markov estimator of the vector $K^\ast \cb$ in model \eqref{eq:regress01} is given 
by the $ s \times s$ matrix 
\begin{equation}
    \label{det1}
  \frac{\sigma^2}{n} K^\ast (M(\mu_n))^{-}K
\end{equation}
where $A^-$ denotes a generalized inverse of the matrix $A$ (see, \citet{Pukelsheim2006}; note that in this reference only real-valued matrices are considered, but results can easily be transferred to the complex-valued case).
It can be shown that under the range inclusion condition \eqref{eq:cond:range:inclusion} the $s \times s$ matrix in \eqref{det1} does not depend on the specific choice of this generalized inverse and has full rank $s$.
An optimal design minimizes (resp. maximizes) a real-valued function of the matrix $K^\ast (M(\mu_n))^{-}K$ (resp. its inverse) leading to non-linear and discrete optimization problems which are unfortunately in general intractable.
A common approach is to relaxate such problems by optimizing the optimality functional on the convex set $\mathscr P_{G}$ of all probability measures on $G$.
Following this approach we define
\begin{equation*}
    \mathcal S_{K}(\phib)
    \defeq \{\mu\in \mathscr P_{G} : \range(K) \subseteq \range(M(\mu)) \}\,,
\end{equation*} 
as the set of all {\it feasible measures} in $\mathscr P_{G}$, and for any $\mu \in \mathcal S_{K}(\phib)$ we consider the matrix
\begin{equation}\label{eq:ck}
    C_{K}(\mu)
    \defeq (K^\ast M(\mu)^{-} K)^{-1}\,.
\end{equation} 
Note that the inverse in this definition exists because the matrix ${K}$ has full column rank and the rank inclusion condition is satisfied.
Moreover, $\mathcal S_{K}(\phib)$ is a compact and convex subset of $\mathscr P_{G}$ and $\mu_G \in \mathcal S_{K}(\phib)$. 
For any fixed $-\infty \leq p < 1$, we define Kiefer's $\Phi_p$-criteria \citep{Kiefer1974} for $-\infty \leq p < 1$ by  
\begin{equation*}\label{eq:via}
    \Phi_p(\mu) \defeq (\trace
    (C_{K}^p(\mu)
    ))^{1/p}, \quad \mu \in \mathcal S_{K}(\phib).
\end{equation*} 
For $p=-\infty$, this definition becomes $\Phi_{-\infty}(\mu) = \lambda_{\min} 
(C_{K}(\mu))$, the minimum eigenvalue of the matrix $C_{K}(\mu)$.
Note that $\Phi_p$ defines a non-negative real-valued function on $\mathcal S_{K}(\phib)$.
A probability measure $\mu_\ast \in \mathcal S_{K}
(\phib)$ is called (approximate) \emph{$\Phi_p$-optimal design} if it maximizes the functional $\Phi_p$ over $\mathcal S_{K}(\phib)$, that is, 
\begin{equation*}\label{eq:problem}
\Phi_p(\mu_\ast)=\max_{\mu\in \mathcal 
S_{K}(\phi)}\Phi_p(\mu)\,
\end{equation*} 
(it can be shown that the maximum exists, see \citet{Pukelsheim2006}).
Additionally, we define a further optimality criterion by 
\begin{equation*}\label{eq:opt2}
    \Phi_{E_s}(\mu)\eqv 
    \sum_{j=1}^{s}\kappa_j\,,
\end{equation*} 
where  
$\kappa_1 \leq \cdots \leq 
\kappa_{D}$ are the increasingly ordered eigenvalues of the positive definite matrix $M(\mu)$.
For any $s = 1,2,\ldots,D$, a probability measure 
$\mu_0\in \mathscr P_{G} (\phib)$ is called 
(approximate) \emph{$\Phi_{E_s}$-optimal design} if it maximizes $\Phi_{E_s}$ over the set $\mathscr P_{G}$, that is,   
\begin{equation*}\label{eq:some1}
\Phi_{E_s}(\mu_\ast)=\operatornamewithlimits{max}_{\mu\in 
        \mathscr P_{G}}
        \Phi_{E_s}(\mu)\,.
\end{equation*}
Optimal designs for these criteria have been investigated in \citet{Harman2004} and \citet{Filova2011}, among others. 

\section{Optimality of the normalized Haar measure}\label{sec:opt2}

Our principal interest is in designs that are particularly efficient for the estimation of the coefficients corresponding to complete blocks of eigenfunctions
\begin{equation*}
    \{ \phi_{j,1},\ldots,\phi_{j,s_j} \}
\end{equation*}
belonging to a specific eigenvalue $\lambda_j$ resp. a specific resolution level.
For this, we consider a tuple $(k_0,k_1,\ldots,k_{q})$ 
of integers satisfying 
\begin{equation*}
    0\leq k_0<k_1<\ldots<k_q\leq d\,,
\end{equation*} 
and for $s \defeq \sum_{i=0}^q s_{k_i}$ we consider the $s \times D$ block matrix 
$K^\ast =(K_{i,j})
_{i=0,\ldots,q}^{j =0,\ldots,d} \in \C^{s\times D}$ 
with blocks $K_{i,j}$ defined by 
\begin{align}\label{eq:k1}
    K_{i,j}=\begin{cases}
    0_{s_{k_i},s_{k_j}} & \text{if } j \neq k_i,\\ 
    I_{s_{k_i}} & \text{if } j = k_i,\\
    \end{cases}   
\end{align}
where $0_{a,b}$ denotes the $a \times b$ matrix with all entries equal to $0$ and $I_a$ the $a \times a$ identity matrix.
Note that $K \in \C^{D \times s}$ is defined such that $K^\ast \cb$ yields the vector with coefficients
\begin{equation*}
    \{ c_{k_i,j} \mid i=0,\ldots,q, j=1,\ldots,s_{k_i} \}
\end{equation*}
in model \eqref{eq:regress01}.
Two cases are particularly interesting. First, for $q=d$, $K$ becomes the $D \times D$ identity matrix and the full parameter vector $\cb$ is the target parameter.
On the other hand, if $q=0$, only the coefficients $c_{k_0,1},\ldots,c_{k_0,s_{k_0}}$ at a certain resolution level $k_0$ (corresponding to the eigenvalue $\lambda_{k_0}$) are considered.
\begin{theorem}\label{thm:optHaar}
Let $G$ be a compact and connected Lie group and for integers $0 \leq k_0 < \ldots < k_q \leq d$ consider the matrix $K$ defined in \eqref{eq:k1}.
Then, the following statements hold:
\begin{enumerate}[label=(\alph*)]
    \item\label{it:optimality:Haar:i} 
    For any $p \in [-\infty,1)$,
    the Haar measure $\mu_{G} $ is $\Phi_p$-optimal for estimating the linear combinations $K^* \boldsymbol{c}$ in the linear regression model \eqref{eq:regress01} with regression functions $\phib$ in \eqref{eq:def:phib}.
    \item\label{it:optimality:Haar:ii} For any $s=1,\ldots,D$, the Haar measure $\mu_{G} $ is $\Phi_{E_{s}}$-optimal for estimating the linear combinations $K^* \boldsymbol{c}$ in the linear regression model \eqref{eq:regress01} with regression functions $\phib$ in \eqref{eq:def:phib}.
\end{enumerate}
Moreover, the corresponding information matrix associated with the Haar measure $\mu_G$ in the considered regression model is given by $M(\mu_G)=I_D$.
\end{theorem}

\begin{proof}
In order to prove  part \ref{it:optimality:Haar:i}, first note that the definition of the information matrix $M(\mu)$ in \eqref{eq:info} together with the orthonormality of the basis functions in $\phib$ with respect to the Haar measure directly implies that $M(\mu_{G}) = I_{D}$.
In order to establish optimality of the Haar measure, we use the general equivalence theorem in Section~7.20 of \citet{Pukelsheim2006}.
For the case $p \in (-\infty,1)$, this result implies that the Haar probability measure $\mu_{G}$ is $\Phi_p$-optimal if and only if the inequality
 \begin{equation}\label{eq:equiv}
    \phib^\ast(g) M(\mu_{G})^- K C_{K}(\mu_{G})^{p+1} K^\ast (M(\mu_{G})^-)^\ast \phib(g) \leq \trace (C_{K}^p(\mu_{G})) 
\end{equation} 
holds for all $g \in G$, where the matrix $C_{K}(\mu_G)$ is defined in \eqref{eq:ck}.
Since $M(\mu_{G}) = I_{D}$, we have $C_{K}(\mu_{G}) = (K^\ast K)^{-1} =I_s$, and using the definition of the matrix $K$ in \eqref{eq:k1}, the right-hand side of~\eqref{eq:equiv} becomes
\begin{equation*}
  \trace(C^p_{K}(\mu_{G})) = \trace ((K^\ast K)^{-p}) = \trace({I}_s) = s.
\end{equation*} 
On the other hand, using $K^\ast K = I_s$, the left-hand side of \eqref{eq:equiv} is equal to
\begin{align*}
    \phib^\ast(g)  K (K^\ast K)^{-p-1} K^\ast \phib(g) &= \phib^\ast(g)  K  K^\ast \phib(g)= \sum_{i=0}^q \sum_{j=1}^{s_{k_i}} \lvert \phi_{k_i,j}(g) \rvert^2 = \sum_{i=0}^q s_{k_i}= s,
\end{align*}
where we used the addition formula $\sum_{j=1}^{s_{k_i}} \lvert \phi_{k_i,j}(g) \rvert^2 = s_{k_i}$ for a complete set of eigenfunctions of the eigenspace $E_{\lambda_{k_i}}$ of the Laplace-Beltrami operator (see \citet{Gine1975}, Theorem 3.2).
This finishes the proof of \ref{it:optimality:Haar:i} for the case $p > - \infty$.

Next we consider the case  $p=- \infty$, where Theorem 7.22 in \citet{Pukelsheim2006} shows that the design $\mu_G$ is $\Phi_{-\infty}$-optimal if and only if there exists a matrix $E$ with $\trace(E)=1$ and a generalized inverse $M^-$ of  the matrix $M (\mu_G)$ such that the inequality 
\begin{align}
\label{det7}
 \phib^\ast(g) M^- K C_{K}(\mu_{G}) E C_{K}(\mu_{G}) 
 K^\ast (M{^-})^\ast \phib(g) \leq \lambda_{\min}  (C_{K}(\mu_{G})) 
\end{align}
holds for all $g \in G$.
Now, for the design $\mu_G$ we have  $C_{K}(\mu_{G}) =I_s$,  $M(\mu_{G}) = I_{D}$, which 
gives $\lambda_{\min}  (C_{K}(\mu_{G})) = 1$ and  $M^-=I_{D}$. Therefore, if we can choose $E$ as the identity matrix, the inequality \eqref{det7} reduces to $\phib^\ast(g)  K  K^\ast \phib(g) \leq 1 $.  By the discussion of the case $p > - \infty $ this is in fact an identity, which proves the statement for the case $p=-\infty$.

In order to prove part \ref{it:optimality:Haar:ii}, we first note that the subgradient 
$\partial \Phi_{E_s}$ of 
$\Phi_{E_s}$ at 
$M(\mu_{G})=I_{D}$ is  given by  
\begin{equation*}
    \partial \Phi_{E_s}(\mu_{G})
    = \{ \diag(\gamma_1,\ldots,\gamma_{D}) \mid 
    \gamma_1,\ldots,\gamma_D \in [0,1],~
    \gamma_1+\ldots+\gamma_{D} = s\}\,,
\end{equation*} where $\diag
(\gamma_1,\ldots,\gamma_{D})$ denotes the $D \times D$ 
diagonal matrix with diagonal $(\gamma_1,\ldots,\gamma_{D})$.
By Theorem~4 in \citet{Harman2004}, the measure $\mu_{G}$ is 
$\Phi_{E_s}$-optimal if and only if there exists a $\Gamma \in \partial \Phi_{E_s}(\mu_{G})$ such that the inequality 
\begin{equation*}
    \phib(g)^\ast \Gamma \phib(g)
    \leq \Phi_{E_s}(\mu_{G})
\end{equation*} 
holds for all $g \in G$.
We consider $\Gamma=K K^\ast$ where the matrix $K$ is defined in \eqref{eq:k1}.  This matrix is an element of the subgradient
because it is diagonal with entries 
in $[0,1]$, and
\begin{equation*}
    \trace(\Gamma)=\trace(KK^\ast)=\trace(K^\ast K)
    =\trace(I_s)=s\,.
\end{equation*} 
Moreover, as seen in the proof of \ref{it:optimality:Haar:i}, one has
\begin{align*}
    \phib^\ast(g) \Gamma \phib(g) =\phib^\ast(g) K
    K^\ast \phib(g) = s =\Phi_{E_s}
    (\mu_{G}) 
\end{align*}
for all $g\in G$, which proves \ref{it:optimality:Haar:ii}.
\end{proof}

\begin{remark}
Both assertions of Theorem~\ref{thm:optHaar} hold true for any  compact homogeneous Riemannian manifold and not only for the special case of compact and connected Lie groups.
In fact, the assumptions on the Lie group $G$ made 
in this paper imply that $G$ is a compact homogeneous Riemannian manifold.
In this general case, the Haar measure $\mu_G$ on $G$ has to be replaced with the normalized volume element associated with the metric tensor $g$ of the manifold.
\end{remark}

\section{$\lambda$-designs, $\Phi_p$- and $\Phi_{E_s}$-optimality}
\label{sec:lambda}

In Theorem~\ref{thm:optHaar} we have identified the Haar measure of a compact and connected Lie group as an optimal approximate design with respect to both the $\Phi_p$-
and the $\Phi_{E_s}$-optimality criteria.
However, this measure cannot be implemented in practice to define a design of a given finite sample size $n$.
As the space $\mathscr M_{G}(\phib)$ in \eqref{det2} is convex 
and compact, it follows from Carath{\'e}odory's theorem that there is always a design $\mu_{\text{finite}}$
with weights $w_1, \ldots , w_{s^\star}$ at pairwise distinct points $g_1,\ldots,g_{s^\star}$, where
\begin{equation}\label{eq:bounds:s:star}
    s \leq s^\star \leq {1 \over 2} s (s+1) + s(D -s)
\end{equation}
such that $ M(\mu_{\text{finite}}) = M(\mu_{G})$
(see \citet{Pukelsheim2006}, Section~8.3).
Thus, if an overall budget of $n$ observations is allowed, one can round off the quantities $nw_i$ to integers, say $n_i$, such that $\sum_{i=1}^{s^\star} n_i =n$, as in \citet{Pukelsheim1992}. Taking $n_1, \ldots , n_{s^\star}$ observations. at the points $g_1, \ldots , g_{s^\star}$, respectively,  then 
yields an implementable design for the %regression 
model \eqref{eq:regress01}.

However, the points $g_i$ and weights $w_i$ have to be found numerically by solving a non-linear optimization problem, that is maximizing the optimality criterion in the class of all designs with at most ${1 \over 2} s (s+1) + s(D -s)$ support points.
For the regression models considered in this paper this number  will be very large.
For example, if all components of the vector 
$\boldsymbol{c}$  are to be estimated, we have $K=I_D$, \eqref{eq:bounds:s:star} reduces to $D \leq s^\star \leq {1 \over 2} D(D+1)$, and  usually $D$ is very large as we will illustrate by the examples below. In such cases, the
resulting  optimization problem is non-linear, high-dimensional and computationally hard to solve (although it is convex). Moreover, the  number $s^*$ of support points of the determined optimal design will be large as well in these cases.
As a consequence, for moderate sample sizes $n$ many weights will be very small such that it is difficult to round $nw_{i}$ to a positive integer. Therefore, the exact designs obtained by this approach are often inefficient.

\begin{example}\label{example:sstar:sphere}
For the case of the sphere $\mathbb S^2$ (which is a homogeneous manifold, but not itself a Lie group), optimal designs for the regression problem \eqref{eq:regress01}, formulated in terms of spherical harmonics as basis functions, have been studied in \citet{Dette2005}.
The spherical harmonics (usually denoted by $Y_{\ell,m}$) form a complete set of eigenfunctions of the Laplace-Beltrami operator for the eigenvalues $\lambda_\ell = \ell(\ell+1)$, $\ell = 0,1,\ldots$ (the stated eigenvalues are valid for the non-normalized sphere of volume $4\pi$).
The eigenspace corresponding to the eigenvalue $\lambda_\ell$ has dimension $s_\ell = 2 \ell+1$, and considering a regression model given by \eqref{eq:regress01} and \eqref{eq:def:phib} in terms of all the spherical harmonics up to the $(d+1)$-st eigenvalue yields overall dimension $D = \sum_{\ell=0}^d (2\ell+1) = (d+1)^2$.
Consequently, for the case $K=I_D$, the general bound becomes
\begin{equation*}
 d^2 \approx (d+1)^2 \leq s^\star \leq \frac{(d+1)^4+(d+1)^2}{2}  \approx {d^4  \over 2 }
\end{equation*}
\end{example}

\begin{example}\label{example:sstar:SO3}
For the case of the rotation group $\SO(3)$, which is a non-abelian Lie group, the basis of eigenfunctions of the Laplace-Beltrami operator is usually given by the so-called
\emph{Wigner~$D$-matrices} $D^\ell_{m,m'}$ where $\ell \in \N_0$ and $m,m' = -\ell,\ldots,\ell$.
The precise definition of the $D^\ell_{m,m'}$ is given in Section~\ref{subsec:example:SO3} where a regression model in terms of these eigenfunctions is considered in detail.
The eigenfunctions $D^\ell_{m,m'}$ form an orthonormal basis for the eigenspace corresponding to the $(\ell+1)$-smallest eigenvalue $\lambda_\ell\propto \ell(\ell+2)$.
Hence, this eigenspace has dimension $(2\ell+1)^2$.
Considering a regression model given by \eqref{eq:regress01} and \eqref{eq:def:phib} in terms of all the Wigner $D$-matrices up to the $(d+1)$-st eigenvalue yields overall dimension $D = \sum_{\ell=0}^d (2\ell+1)^2 = (2d+1)(2d+2)(2d+3)/6$.
For the case $K=I_D$, the general bound becomes
\begin{align*}
 &   {4d^3  \over 3 } \approx 
     {(2d+1)(2d+2)(2d+3) \over 6  } \leq s^\star  \\ 
     &  ~~~~~~~~~~~~~~~~~~~~~~~~~~~~~~~~
     \leq \frac{(2d+1)(2d+2)(2d+3) [(2d+1)(2d+2)(2d+3) + 6]}{72}  \approx {8d^6  \over 9 }
     \end{align*}
\end{example}

\begin{example}\label{example:sstar:padhorny}
Let us now consider the product manifold 
$$\mathrm{SO}(3)/ \mathbb S^1 \times \mathrm{SO}(3) \simeq \mathbb S^2 \times \mathrm{SO}(3)$$
considered in \citet{Padhorny2016}.
The natural way to construct a design on this Cartesian product is to take a product design of two designs on the components $\mathbb S^2$ and $\SO(3)$.
The respective values of $s^\star$ for designs on these two manifolds have been constructed in Examples~\ref{example:sstar:sphere} and \ref{example:sstar:SO3}, respectively.
The specialization of the general bounds for $s^\star$ is obtained by multiplying the lower resp. upper bounds from Examples~\ref{example:sstar:sphere} and \ref{example:sstar:SO3}.
Hence, the lower bound on the value of $s^\star$ is
\begin{equation*}
     (d+1)^2 \cdot {(2d+1)(2d+2)(2d+3) \over 6} \approx {4d^5  \over 3 } \,,
\end{equation*}
and the upper bound becomes
\begin{equation*}
    \frac{(d+1)^4+(d+1)^2}{2} \cdot \frac{(2d+1)(2d+2)(2d+3) [(2d+1)(2d+2)(2d+3) + 6]}{72}
    \approx {4d^{10}  \over 9 }\,.
\end{equation*}
\end{example}

In this section we provide an alternative approach for the construction of implementable designs.
For this purpose we consider the concept of $\lambda$-designs (see \citet{Gariboldi2021} for a recent reference in the context of Riemannian manifolds using a slightly different definition).
$\lambda$-designs are generalizations of spherical $t$-designs which have found considerable attention in the literature 
\citep{Delsarte1977,BANNAI20091392,Bondarenko2013}.
Such designs are of interest as they provide quadrature formulas for numerical integration (see, for instance, \citet{Korevaar1994} for more details).

\begin{definition} 
Let $E_{\lambda_j}$ be the eigenspace corresponding to the $j$-th smallest eigenvalue $\lambda_j$ of the Laplace-Beltrami operator.
For any $\lambda>0$, a finite non-empty subset $\operatorname{S}
\subset G$ is called a \emph{$\lambda$-design} if for every 
\begin{equation*}
    f \in \bigoplus_{\lambda_j
        \leq\lambda} E_{\lambda_j}
\end{equation*} 
the following quadrature formula holds: 
\begin{equation}\label{eq:design1}
    \int_{G}f(g) \dd \mu_{G}(g)
    = \frac 1{\lvert S \rvert}
    \sum_{g\in S}
    f(g)\,.
\end{equation} 
\end{definition} 

The following proposition provides an alternative characterization of $\lambda$-designs.

\begin{proposition}\label{prop:characterisation:lambda:design}
A finite non-empty subset $S \subset G$ is a $\lambda$-design if and only if the equality 
\begin{align}\label{eq:root}
    \frac 1{\lvert S \rvert}
    \sum_{g\in S} f(g) =0
\end{align}
holds for all $f \in \bigoplus_{0<\lambda_j \leq\lambda} E_{\lambda_j}$.   
\end{proposition}
    
\begin{proof}
Note that $\lambda_0=0$ so that 
$E_{\lambda_0}$ consists of the constant functions on $G$, for which identity \eqref{eq:design1} obviously holds.
Moreover, by linearity, it is sufficient to consider \eqref{eq:root} for eigenfunctions $\phi$ in some eigenspace $E_{\lambda_j}$ only.
First, assume that \eqref{eq:design1} holds and let $\phi \in E_{\lambda_j}$ for some $0 < \lambda_j \leq \lambda$.
Then,
\begin{align*}
    \frac 1{\lvert S \rvert} \sum_{g\in S}\phi(g) = \int_{G}\phi(g) \dd \mu_{G}(g) = \int_{G} (\phi(g) \cdot 1) \dd \mu_{G}(g) = 0
\end{align*}
by the orthogonality of eigenspaces, which yields \eqref{eq:root}.
Vice versa, assuming that \eqref{eq:root} holds, we write an arbitrary $f \in \bigoplus_{\lambda_j \leq\lambda} E_{\lambda_j}$ as $f=f_1 + f_2$ for some constant function $f_1$ and $f_2 \in \bigoplus_{0<\lambda_j \leq\lambda} E_{\lambda_j}$.
Then,
\begin{align*}
    \int_{G}f(g) \dd \mu_{G}(g) = \int_{G}f_1(g) \dd \mu_{G}(g) = 
    { 1 \over |S]} \sum_{g\in S} f_1(g) = { 1 \over |S]} \sum_{g\in S} f(g)
\end{align*}
where \eqref{eq:root} was used to justify the last step.
Hence, \eqref{eq:design1} holds which finishes the proof.
\end{proof}

At first sight, the notions of $\Phi_p$-, $\Phi_{E_s}$-optimality and $\lambda$-designs seem to be unrelated.
However, as we will prove in the following Theorem~\ref{thm:lambda:Phi}, $\lambda$-designs are $\Phi_p$- and $\Phi_{E_s}$-optimal. 

\begin{theorem}\label{thm:lambda:Phi}
    Let $G$ be a compact and connected Lie group.
    Consider the regression model defined by \eqref{eq:regress01} and \eqref{eq:def:phib}, that is, the linear regression model is defined in terms of a basis
    of $D$ orthonormal eigenfunctions corresponding to the first $d+1$ eigenvalues $\lambda_0,\ldots,\lambda_d$ of the negative Laplace-Beltrami operator $-\Delta_{G}$.
    Suppose that the matrix $K$ is defined as in \eqref{eq:k1}.
    If $\lambda_d$ is sufficiently large, then for
    every $14\lambda_d$-design $S \subset G$, the empirical measure
    \begin{equation*}
    \mu_{S} \defeq \frac 1{\lvert S \rvert} \sum_{g \in S} \delta_{g}
\end{equation*}
is $\Phi_p$-optimal  (for every $-\infty \leq p < 1$) and $\Phi_{E_s}$-optimal (for every $s = 1,2,\ldots,D$)
for estimating the linear combinations $K^\ast \cb$.
\end{theorem}

\begin{proof}
As the proof of this theorem is 
slightly technical, we present a skeletal proof here, and 
refer the reader to \cref{app:proof:details} for 
supplementing details.
For $g \in G$, set
\begin{equation*}
    \phib^\top(g) = (\phi_{0,1}(g),\phi_{1,1}(g),\ldots,\phi_{1,s_1}(g),\ldots,\phi_{d,1}(g),\ldots,\phi_{d,s_d}(g)),
\end{equation*}
where for any fixed $j$ the functions $\phi_{j,1},\ldots,\phi_{j,s_j}$ form  a complete set of orthonormal eigenfunctions of the Laplace-Beltrami operator associated with the eigenvalue $\lambda_j$.
Let $S \subset G$ be an arbitrary $14\lambda_d$-design.
It will be sufficient to show that the information matrix $M(\mu_S)$ associated with $\mu_S$ is equal to the identity matrix $I_D$, which implies both $\Phi_p$- and $\Phi_{E_s}$-optimality using  the same arguments as given in the proof of Theorem~\ref{thm:optHaar}.
The matrix 
$M(\mu_S)$ has entries
\begin{equation*}
    (M(\mu_S))_{k,\ell} =\frac 1{\lvert S \rvert} \sum_{g \in S} (\phib(g) \phib^\ast(g))_{k,\ell}\,.
\end{equation*}
In particular, for $k \neq \ell$, it follows from Equation~\eqref{eq:product:nu:neq:mu} in  Appendix~\ref{app:proof:details}  that the entry $(\phib(\cdot) \phib^\ast(\cdot))_{k,\ell}$, which is the product of a eigenfunction with the complex conjugate of a different eigenfunction, belongs to the space $\bigoplus_{0 < \lambda \leq 14 \lambda_d} E_\lambda$ provided  that $\lambda_d$ is sufficiently large (for the rigorous meaning of \emph{sufficiently large} we refer to Appendix~\ref{app:proof:details}). Using the characterization of a $14\lambda_d$-design in  Proposition~\ref{prop:characterisation:lambda:design} this implies that
\begin{equation}\label{eq:M:0}
    (M(\mu_S))_{k,\ell} = \frac 1{\lvert S \rvert}
    \sum_{g \in S} (\phib(g) \phib^\ast(g))_{k,\ell}=0 \,.
\end{equation}
Similarly, again assuming that $\lambda_d$ is sufficiently large, Equation~\eqref{eq:product:nu:eq:mu} from Appendix~\ref{app:proof:details} implies that $(\phib(\cdot) \phib^\ast(\cdot))_{k,k}$, which is now the product of an eigenfunction with its complex conjugate, can be written in the form $1+\phi$ with $\phi = \bigoplus_{0 < \lambda \leq 14 \lambda_d} E_\lambda$. 
 Using again the characterization of a $14\lambda_d$-design in Proposition~\ref{prop:characterisation:lambda:design},  this yields  
\begin{equation}\label{eq:M:1}
    (M(\mu_S))_{k,k} = 1
\end{equation}
Thus, combining \eqref{eq:M:0} and \eqref{eq:M:1} implies $M(\nu)=I_D$, which finishes the proof of the theorem.
The remaining details regarding Equations~\eqref{eq:product:nu:neq:mu} and \eqref{eq:product:nu:eq:mu} are given in Appendix~\ref{app:proof:details}.
\end{proof}

\begin{remark}
Theorem~\ref{thm:lambda:Phi} relates the $\Phi_p$- and $\Phi_{E_s}$-optimal design problem to the problem of finding  $\lambda$-designs.
The existence of these designs, and hence $\Phi_{E_s}$- and $\Phi_p$-optimal designs, is guaranteed by \citet{Gariboldi2021}, Theorem~6, which implies that for any integer $N$ satisfying
\begin{equation*}\label{eq:N}
N \geq C_G \lambda^{\dim G}
\end{equation*}
with $C_G$ being a suitable constant (that depends only on the Lie group $G$) there exists a $14\lambda$-design on $G$ of cardinality $N$.
In Section~\ref{sec:examples} we will give examples of $\Phi_p$- and $\Phi_{E_s}$-optimal designs for small sample sizes.
\end{remark}

\begin{remark}
    Note that, in contrast to the method of construction mentioned at the beginning of this section, the empirical measure $\mu_S$ has equal weights.
    Hence, $\mu_S$ is directly implementable as an exact  design of cardinality $\lvert S\rvert$ which is $\Phi_p$-optimal.
    In particular, no loss in efficiency due to an additional rounding procedure occurs.
    The examples in Section~\ref{sec:examples} will show that for classical groups like $\SU(2)$ and $\SO(3)$ one does not even have to construct a $14\lambda_d$-design (the numerical constant $14$ being an artefact of our method of proof which holds for any compact and connected Lie group) if $\lambda_d$ is the maximal frequency in the considered regression model.
    Hence, at least for the special cases considered there, $S$ can be chosen as a $\lambda$-design for some $\lambda < 14\lambda_d$ potentially resulting in an even smaller optimal design which provides a further improvement in comparison to the general theory.
    The arguments for this improvement are related to the classical Clebsch-Gordan formulae for $\SU(2)$ and $\SO(3)$, respectively, which describe the irreducible components of the tensor product of two irreducible representations.
    \end{remark}

\begin{remark}\label{rem:rho}
    The assumption of Theorem~\ref{thm:lambda:Phi} that the $d+1$-st smallest eigenvalue $\lambda_d$ is sufficiently large is satisfied if the maximal resolution level in the regression model is chosen large enough because the sequence of eigenvalues of the Laplace-Beltrami operators tends to infinity.
    Moreover, the proof of the theorem, which hinges on results from Lie group theory, even yields a precise statement:
    It is sufficient to assume that $\lambda_d \geq \lVert \rho \rVert \vee 1$ where $\rho=\rho_G$ is a quantity that only depends on the Lie group $G$ (in fact, it depends only on the Lie algebra of the group).
    For instance, the Lie algebras of $\SU(2)$ and $\SO(3)$ are isomorphic, and the quantity $\lVert \rho \rVert$ is equal to $\frac{1}{2}$ in both cases.
    In general, $\rho$ is defined as half the sum of the positive roots with respect to a maximal torus of $G$, see Appendix~\ref{app:proof:details} for details.
    Because the quantity $\lVert \rho \rVert$ is quite small for compact matrix groups of low dimension, this assumption is not very restrictive in practice.
    Ad hoc arguments, which are again essentially based on the classical Clebsch-Gordan formula, show that for the groups $\SU(2)$ and $\SO(3)$ considered as examples in Section~\ref{sec:examples} below this additional assumption can be dropped at least in these special cases.
\end{remark}

\begin{remark}
Note that $\lambda$-designs are optimal designs for the two-dimensional sphere $\mathbb S^2 \subset \R^3$ (which is not a Lie group).
This result for $\mathbb S^2$ is a direct consequence of the result for the rotation group $\SO(3)$ since spherical harmonics are related to the Wigner-$D$ functions.
However, this fact can also be checked by direct calculations similar to the ones in \citet{Dette2005}.
For the spheres $\mathbb S^1$ and $\mathbb S^3$, which are the only Lie groups among spheres, optimal designs can be derived directly from the general Theorem~\ref{thm:lambda:Phi}.
The case of the one-dimensional sphere, which is of course well-known, arises as the easiest special case of the tori discussed in Section~\ref{subsec:example:tori} below.
\end{remark}

\section{Examples of $\Phi_p$- and $\Phi_{E_s}$- optimal designs for linear regression on Lie groups}\label{sec:examples} 

In this section, we illustrate the general procedure of constructing $\Phi_p$- and $\Phi_{E_s}$-optimal designs as $\lambda$-designs by concrete examples.
Note that in the concrete examples we are able to improve on the statement of the general Theorem~\ref{thm:lambda:Phi} by exploiting \emph{ad hoc} properties of the specific Lie groups under consideration.
In all the specific examples considered we will exploit the special knowledge that the product of eigenfunctions up to the chosen resolution level can expressed as the linear combination of the eigenfunctions up to a (slightly higher) resolution level.
First, we recap how the example of tori, that is, compact, connnected, and commutative Lie groups fits into our general theory.
The case of the circle group is discussed beforehand.
Then, we consider the problem of constructing $\Phi_p$-optimal designs for 
regression problem on the Lie groups $\SU(2)$ and $\SO(3)$, respectively.
We close by revisiting the example of the product manifold $\mathbb S^2 \times \SO(3)$ which is motivated by the application from biology that was already discussed in the introduction.

\subsection{Optimal designs on the circle group}\label{subsec:example:circle}

Our first example is used  to illustrate how known results about optimal designs on the circle group $\mathbb S^1$ relate in comparison to the general theory presented in the previous sections.
We start our discussion with the case of the circle group $\mathbb S^1 \cong \R/ \Z \cong (0,1]$.
We consider the regression problem
\begin{equation*}
     \Eb [Y|g] = \phib^\top(g) \boldsymbol{c}, \quad g \in (0,1]\,,
\end{equation*}
in terms of the (complex) Fourier basis which is given by
\begin{equation*}
    \phib(g) = \big ( 1, \ee^{2\pi \ii g}, \ee^{-2\pi \ii g},\ldots,\ee^{2\pi n \ii g}, \ee^{-2\pi n \ii g}\big )^\top  \in  \mathbb{C}^{2n+1}\,.
\end{equation*}
Writing out the regression equation, the model can thus be written as
\begin{equation*}
    \Eb [Y|g] = \sum_{k = n}^{n} c_k \ee^{2\pi \ii k g}\,.
\end{equation*}
Note that the eigenvalues of the Laplace-Beltrami operator here satisfy $\lambda_k \asymp k^2$ for $k \in \Z$.
Any product of two eigenfunctions with eigenvalues $\lesssim n^2$ can be expressed as the linear combination of eigenfunctions with eigenvalues $\lesssim (2n)^2$.
Hence, on needs to construct a $\lambda$-design for $\lambda \asymp (2n)^2$.
It is well known that such a design $\mu$ is obtained by putting equal weights to the points $\frac{i}{2n+1}$, $i=1,\ldots,2n+1$ which follows from the algebraic identity
\begin{equation*}
    \sum_{s=1}^{2n+1} \ee^{\frac{2 \pi \ii k s}{2n+1}} = \begin{cases}
       0 , & \text{ %falls
       {if} } k \neq 0,\\
       2n+1, & \text{ %falls
       {if} } k \in \{ -n,\ldots,n \} \setminus \{ 0 \}\,.
    \end{cases}
\end{equation*}
Hence, the design $\mu= \frac{1}{2n+1} \sum_{s=1}^{2n+1} \delta_{\frac{s}{2n+1}}$ is $\Phi_p$-
and $\Phi_{E_s}$-optimal for the regression model containing the $2n+1$ eigenfunctions up to the eigenvalue $\lambda_{n} \asymp n^2$.
In this situation, statement of the general Theorem~\ref{thm:lambda:Phi} would suggest to use a $\lambda$-design with $\lambda \asymp k^2$ for $k = \lceil \sqrt{14}n \rceil$.

\subsection{Optimal designs on tori $\mathbb T^d \cong \mathbb S^1 \times \ldots \mathbb S^1$}\label{subsec:example:tori}

The previous example of the circle group $\mathbb S^1$ can be generalized to the case of tori $\mathbb T^d \cong \mathbb S^1 \times \ldots \times \mathbb S^1 \cong (0,1]^d$ of arbitrary dimension (in essence, the tori are the compact, connected and commutative Lie groups).
The canonical basis for the regression problem in this case contains the basis functions of the form
\begin{equation}\label{eq:eigenfct:torus}
    (0,1]^d \mapsto \C, \quad (x_1,\ldots,x_d) \mapsto \ee^{2\pi \ii k_1 x_1}\cdot \ldots \cdot \ee^{2\pi \ii k_d x_d} = \ee^{2\pi \ii k_1 x_1 + \ldots + 2\pi \ii k_d x_d}\,,
\end{equation}
where $k=(k_1,\ldots,k_d) \in \Z^d$ is a multi-index.
Note that these functions form a (complete) system of normalized eigenfunctions of the Laplace-Beltrami operator.
More precisely, the function in \eqref{eq:eigenfct:torus} is an eigenfunction with eigenvalue $\lambda_k \asymp k_1^2 + \ldots + k_d^2$.
It is easily checked that the product of two eigenfunctions with multiindices $k=(k_1,\ldots,k_d)$ and $\ell = (\ell_1,\ldots,\ell_d)$, respectively, is an eigenfunction with eigenvalue $\lambda_{k + \ell} \leq 2\lambda_k + 2\lambda_\ell \leq 4 \max \{ \lambda_k,\lambda_\ell \}$.
Consequently, if the regression problem \eqref{eq:regress01} is stated in terms of the eigenfunctions up to a certain resolution level $\lambda$, then it is sufficient to construct a $4\lambda$-design which will then be $\Phi_p$-
and $\Phi_{E_s}$-optimal.
Also in this case, the constant $14$ in the statement of the general Theorem~\ref{thm:lambda:Phi} can be reduced.

\subsection{Optimal designs on $\SU(2)$}\label{subsec:example:SU2} We now discuss optimal designs for linear regression on the special unitary group $\SU(2)$ as a first example of a non-abelian Lie group.

\subsubsection{Preliminaries}  
Recall that the real Lie group $\SU(2)$ is defined as
\begin{align*}
    \SU(2) &= \{ A \in \GL(2,\C) \mid A^\ast A = A A^\ast = I_2 \}\\
    &= \left\lbrace  \begin{pmatrix}
    \alpha & \beta\\
    -\bar \beta & \bar \alpha
    \end{pmatrix}   \in \C^{2 \times 2} \mid \lvert \alpha \rvert^2 + \lvert \beta \rvert^2 = 1 \right\rbrace.
\end{align*}
Alternatively, using the identifications $\alpha = a+d \imath$ and $\beta = b+c\imath$, the group $\SU(2)$ can be described as the set of unit quaternions,
\begin{equation*}
    \{ a1 + bi + cj +dk \mid a,b,c,d \in \R \text{ with } a^2+b^2+c^2+d^2=1 \}\,,
\end{equation*}
together with the well-known rules for the quaternion product, that is, $i^2 = j^2 = k^2 = -1$, $ij=k$, $jk=i$, $ki=j$, $ji=-k$, $kj=-i$, and $ik=-j$.
The unit quaternions may be identified with the three-dimensional sphere $\mathbb S^3 \subset \R^4$.
The irreducible representations of $\SU(2)$ are given by the actions $\pi_m$ of $\SU(2)$ on the vector spaces $V_m \subset \C[X,Y]$ of homogeneous polynomials of degree $m$,
\begin{equation*}
    %f \in 
    V_m\ni f \mapsto \pi_m f \in V_m \text{ with } (\pi_m(A) f)(X,Y) = f((X,Y) A) \text{ for } A \in \SU(2)\,.
\end{equation*}
For any $m \in \N_0$, the dimension of $\pi_m$ is $m+1$.
Using the mentioned identification of $\SU(2)$ with $\mathbb S^3$, the corresponding matrix coefficients can be identified with the space hyperspherical harmonics of order $m$, which is a subspace of dimension $(m+1)^2$ of $L^2(\mathbb S^3)$.
These subspaces in turn can be identified with the eigenspaces of the Laplace-Beltrami operator.
The eigenvalue $\lambda_k$ corresponding to the irreducible representation $\pi_k$ is $\lambda_k = k(k+2)$ for $k \in \N_0$.

\subsubsection{An explicit example (spherical $2$-design with $n=5$ points on $\mathbb S^3$)}\label{subsubsec:SU2:explicit}
Now assume that we consider the regression model~\eqref{eq:regress01} where the regression function is assumed to be a linear combination of the eigenfunctions corresponding to the first two eigenvalues $\lambda_0 = 0$ and $\lambda_1 = 3$ of the Laplace-Beltrami operator.
Recall from Remark~\ref{rem:rho} that $\lambda_1=3$ for the results of Theorem~\ref{thm:lambda:Phi} to hold since $\rho=\frac{1}{2}$ in this case.
The regression model can be thus be written as
\begin{equation*}
    \Eb [Y|g] = \sum_{\ell=0}^1 \sum_{\mu_1=0}^\ell \sum_{\mu_2=-\mu_1}^{\mu_1} c_{\ell, \mu_1,\mu_2} Y_{\ell,\mu_1,\mu_2}(g)\,,\end{equation*}
where the $Y_{\ell,\mu_1,\mu_2}$ denote the %so-called 
\emph{hyperspherical harmonics} (see \citet{Avery1982} for a definition of these functions; the argument $g$ is usually parameterized by polar coordinates).
In particular, the number of regression functions in this model is equal to $1+4=5$.
As the proof of Theorem~\ref{thm:lambda:Phi} shows one has to construct a $\lambda_k$-design for an eigenvalue $\lambda_k$ such that products of linear combinations of eigenfunctions with eigenvalues $\leq \lambda_1$ can be expressed as linear combinations of eigenfunctions with eigenvalues $\leq \lambda_k$. 
The general Theorem~\ref{thm:lambda:Phi} states that this holds for any $14\lambda_1$-design which consequently is $\Phi_p$- and $\Phi_{E_s}$-optimal for any $-\infty \leq p < 1$.
Since the strictly increasing sequence of eigenvalues is given by $\lambda_k = k(k+2)$ for the sphere $\mathbb S^3$ of (non-normalized) volume $2\pi^2$, a $14\lambda_1$-design for this case is a $\lambda_5$-design (since $\lambda_5=35\leq 14 \cdot \lambda_1 < 48 = \lambda_6$).
However, in the special case of the special unitary group $\SU(2)$, this result can be improved by the well-known Clebsch-Gordon formula for $\SU(2)$ (see Corollary 5.6.2 in \citet{Kowalski2014}). The Clebsch-Gordon formula states that the tensor product $\pi_1 \otimes \pi_1$ decomposes as 
\begin{equation*}
    \pi_1 \otimes \pi_1 \cong \pi_2 \oplus \pi_0\,.
\end{equation*}
As a consequence, the product of linear combinations of eigenfunctions with eigenvalues $\leq \lambda_1$ can be expressed as a linear combination of eigenfunctions for eigenvalues $\leq \lambda_2$.
Hence, it is sufficient to construct a $\lambda_2$-design in this special case (instead of a $\lambda_5$-design as suggested by the general theorem).

Thanks to the identification $\SU(2) \cong \mathbb S^3$ one can rely on well-known results on $\lambda$-designs for spheres (often referred to as $t$-designs in this case with the correspondance that $\lambda_k$-designs in our notation correspond to $t$-designs with $t=k$).
It is known for the case of the sphere $\mathbb S^3$ that the cardinality $n$ of a spherical $t$-design satisfies the bound
\begin{equation}\label{eq:lower:bound:t_design}
    n \geq \begin{cases}
        2 \binom{\lfloor \frac{t}{2} \rfloor + 3}{3}, & \text{ if } t \text{ is odd,}\\[3mm]
        \binom{\frac{t}{2} + 3}{3} + \binom{\frac{t}{2} + 2}{3},& \text{ if } t \text{ is even}\,,
    \end{cases}
\end{equation}
(this bound is the special case of Equation~(1.3) in \citet{Bajnok1991} for $d=4$), and a $t$-design $X$ is called \emph{tight} if equality holds in \eqref{eq:lower:bound:t_design} for $n=\lvert X \rvert$.
It was proven in \citet{Bannai1979} and \citet{Bannai1980} that tight $t$-designs exist for $\mathbb S^3$ (and any other sphere $\mathbb S^{d-1}$ with $d \geq 3$) if and only if $t \in \{ 1,2,3,4,5,7,11 \}$. 
For example, in our case of interest ($t=2$), it follows that a tight $t$-design with $n=5$ points exists.

An explicit construction of \emph{tight} $t$-designs on $\mathbb S^3$ for $t=2$ is given in \citet{Mimura1990} which we make precise now for $n=5$.
Take the primitive $5$-th root of unity $\zeta = \exp(2\pi \ii/5)$.
For $k=1,2$, define $c_k$ and $s_k$ in $\R^5$ through
\begin{equation*}
    c_k + \ii s_k = {\frac{1}{\sqrt 2}} (\zeta^k,\zeta^{2k},\ldots,\zeta^{nk})^\top\,.
\end{equation*}
Then, the five columns of the matrix $D \in \R^{4\times 5}$ with rows $c_1,s_1,c_2,s_2$ form a $2$-spherical design of minimal cardinality $n=5$.
The matrix $D$ is given by
\begin{align*}
    D = \frac{1}{\sqrt{2}}&\begin{pmatrix*}[r]
        \cos(2\pi/5)\phantom{--} & \cos(4\pi/5)\phantom{--} & \cos(6\pi/5)\phantom{--} & \cos(8\pi/5)\phantom{--} & \cos(2\pi)\\
        \sin(2\pi/5)\phantom{--} & \sin(4\pi/5)\phantom{--} & \sin(6\pi/5)\phantom{--} & \sin(8\pi/5)\phantom{--} & \sin(2\pi)\\
        \cos(4\pi/5)\phantom{--} & \cos(8\pi/5)\phantom{--} & \cos(12\pi/5)\phantom{--} & \cos(16\pi/5)\phantom{--} & \cos(4\pi)\\
        \sin(4\pi/5)\phantom{--} & \sin(8\pi/5)\phantom{--} & \sin(12\pi/5)\phantom{--} & \sin(16\pi/5)\phantom{--} & \sin(4\pi)
    \end{pmatrix*}\\
    \approx &\begin{pmatrix*}[r]
        0.22\phantom{--} & -0.57\phantom{--} & -0.57\phantom{--} & 0.22\phantom{--} & 0.71\\
        0.67\phantom{--} & 0.42\phantom{--} & -0.42\phantom{--} & -0.67\phantom{--} & 0\\
        -0.57\phantom{--} & 0.22\phantom{--} & 0.22\phantom{--} & -0.57\phantom{--} & 0.71\\
         0.42\phantom{--} & -0.67\phantom{--} & 0.67\phantom{--} & -0.42\phantom{--} & 0
    \end{pmatrix*}\,,
\end{align*}
with the running convention that all values are rounded to two decimals.
Using the identification of the three dimensional sphere $\mathbb S^3$ and $\SU(2)$ from Section~\ref{subsec:example:SU2}, we obtain the following five design points on $\SU(2)$:
\begin{equation*}
    \begin{pmatrix}
        \cos(2k\pi/5) + \ii \sin(4k\pi/5)\phantom{--} & \sin(2k\pi/5) + \ii \cos(4k\pi/5)\\
        -\sin(2k\pi/5) + \ii \cos(4k\pi/5)\phantom{--} & \cos(2k\pi/5) - \ii \sin(4k\pi/5)
    \end{pmatrix}\,, \qquad k = 1,\ldots,5\,,
\end{equation*}
or, explicitly:
\begin{align*}
&\begin{pmatrix*}[r]
    0.22+0.42\ii\phantom{--} & 0.67-0.57\ii\\
    -0.67-0.57\ii\phantom{--} &  0.22-0.42\ii
\end{pmatrix*}\,, \quad \begin{pmatrix*}[r]
    -0.57-0.67\ii\phantom{--} & 0.42+0.22\ii\\
  -0.42+0.22\ii\phantom{--} & -0.57+0.67\ii
\end{pmatrix*}\,,\\
&\begin{pmatrix*}[r]
    -0.57+0.67\ii\phantom{--} & -0.42+0.22\ii\\
    0.42+0.22\ii\phantom{--} & -0.57-0.67\ii
\end{pmatrix*}\,, \quad \begin{pmatrix*}[r]
    0.22-0.42\ii\phantom{--} & -0.67-0.57\ii\\
    0.67-0.57\ii\phantom{--} & 0.22+0.42\ii
\end{pmatrix*}\,,\\
&\begin{pmatrix*}[r]
    0.71\phantom{--} & 0.71\ii\\
    0.71\ii\phantom{--} & 0.71
\end{pmatrix*}\,,
\end{align*}
and once again, real and imaginary part of all entries are rounded to two decimals.

\subsubsection{Another explicit example}\label{subsubsec:another}
We briefly report on another explicit example that has been studied in the literature on spherical $t$-designs (see \citet{Bajnok1991} for a more detailed account).
Assume that we want to consider the regression model~\eqref{eq:regress01} where the regression function is assumed to be a linear combination of the eigenfunctions corresponding to the first five eigenvalues $\lambda_0 = 0$, $\lambda_1 = 3$, $\lambda_2=8$, $\lambda_3=15$, and $\lambda_4=24$ of the Laplace-Beltrami operator.
In analogy with the previous example, the model can be written as
\begin{equation}\label{eq:model:4:S3}
    \Eb [Y|g] = \sum_{\ell=0}^4 \sum_{\mu_1=0}^\ell \sum_{\mu_2=-\mu_1}^{\mu_1} c_{\ell, \mu_1,\mu_2} Y_{\ell,\mu_1,\mu_2}(g)\,.\end{equation}
Altogether, the model is stated in terms of $55$ eigenfunctions.
 Theorem~\ref{thm:lambda:Phi} suggests the  construction of a $\lambda_{17}$-design in order to obtain $\Phi_p$-  and $\Phi_{E_s}$-optimality. In this example   we can  apply Clebsch-Gordon theory on $\SU(2)$ which  tells us that it is sufficient to construct a $\lambda_8$-design for this purpose.
For the regression model \eqref{eq:model:4:S3}, an approximate $\Phi_p$- and $\Phi_{E_s}$-optimal design with unequal weights has been proposed in Example~3.1 of \citet{Dette2019}.
The number of design points in this non-uniform  approximate 
optimal design, which is of product type, is $5\cdot 5 \cdot 9=225$.
The existence of a spherical $t$-design for $t=8$ (or equivalently, a $\lambda_8$-design) of cardinality $n$ is guaranteed for all sufficiently large $n$.
Equation~(1.3) in \citet{Bajnok1991} states that such an $n$ must be $\geq 55$ (in fact, even $>$ must hold here since no tight spherical $t$-designs exist for $t=8$).
From a practical point of view it might be sufficient to consider point sets that are numerical spherical designs within the limits of numerical precision only (see \citet{Womersley2018} for further details).
Under
\begin{center}
    \url{https://web.maths.unsw.edu.au/~rsw/Sphere/EffSphDes/S3SD.html},
\end{center}
which from now on is cited as \citep{Womersley2017}, such numerical spherical designs of moderate cardinality have been made available.
For instance, for our example of interest ($t=8$) a numerical spherical design, say  $\mu_{\text{num}}$, consisting of $n=97$ sampling points has been found (file \texttt{sdf008.00097} from \citep{Womersley2017}).
The corresponding exact design has equal weights $1/n = 1/97$ at its support points and the resulting information matrix is numerically close to the identity matrix $I_{55}$.
We now use the rounding procedure considered in \citet{Pukelsheim1992} 
to  construct an exact design, say $\mu_{\text{rounded}}$, with sample size $n=97$ from the   design with $225$ support points constructed in Example~3.1 of \citet{Dette2019}.
Next, we  compare the resulting design with the design obtained by our approach 
via its efficiency defined as
\begin{equation}\label{eq:efficiency:S3}
    \mathrm{eff}_\Phi (\mu_{\text{rounded}}) = \frac{\Phi(\mu_{\text{rounded}})}{\Phi(\mu_{\text{num}})}
\end{equation}
where $\Phi$ can be any of the optimality criteria introduced in Section~\ref{sec:regression}.
Exemplary, we consider the $\Phi_{E_s}$-criteria for $s=1,\ldots,55$ and the $\Phi_p$-criteria for $-2 \leq p < 1$. In 
Figure~\ref{fig:eigenvalues:S3} we show  the eigenvalues of the information matrix $M(\mu_{\text{num}})$ of  the numerical design $\mu_{\text{num}}$ derived from the approximate optimal designs in \cite{Dette2019}
(left part) and of the information matrix $M(\mu_{\text{rounded}})$ for the rounded design $\mu_{\text{rounded}}$ obtained by our approach (right part).
For example, the smallest eigenvalues of the matrix $M(\mu_{\text{rounded}})$ are (numerically) close to $0$, whereas all eigenvalues of $M(\mu_{\text{num}})$ are close to $1$. In particular, the information matrix $ M(\mu_{\text{rounded}})$ is nearly singular and therefore the $97$-point design derived 
from the design proposed by \cite{Dette2019} cannot be used for estimating all parameters in the model.

\begin{figure}[h]
    \centering
%    \begin{subfigure}[t]{0.49\textwidth}
         \centering
         \includegraphics[height=6.5cm,width=6.5cm]{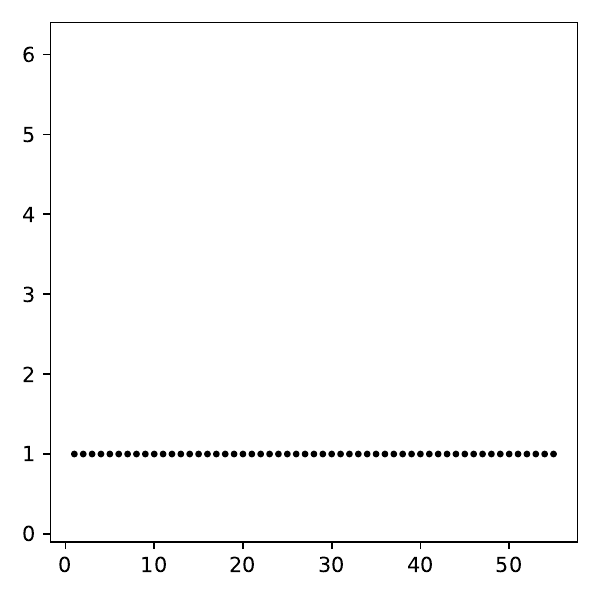}
\includegraphics[height=6.5cm,width=6.5cm]{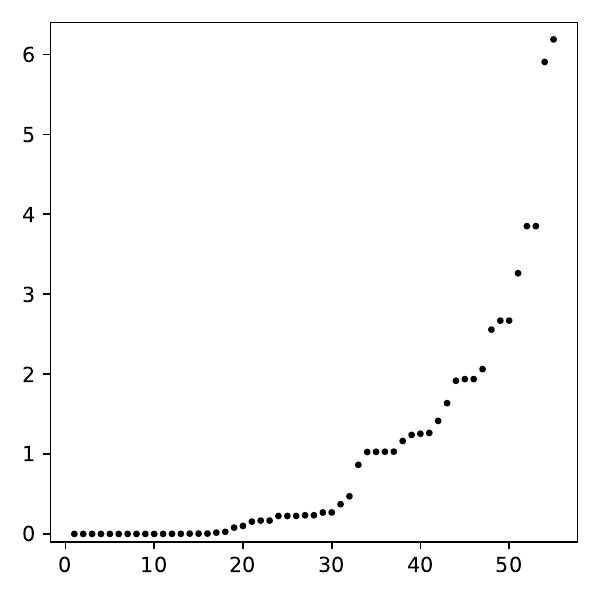}
     \caption{\it Eigenvalues of the information matrices in model \eqref{eq:model:4:S3} for the designs $\mu_{\text{num}}$ (left part) and $\mu_{\text{rounded} }$ (right part).}\label{fig:eigenvalues:S3}
\end{figure}

In Figure~\ref{fig:efficiency:S3} we display the efficiency of the design $\mu_{\text{rounded}}$ with respect to both the $\Phi_{E_s}$ criteria for $s=1,\ldots,55$ (left part) and with respect to Kiefer's $\Phi_p$-criteria for the range $-2\leq p < 1$ (right part).
For the latter, we consider the matrix $K$ corresponding to the estimation of only the $14$ coefficients associated with the three smallest eigenvalues $\lambda_0 = 0$, $\lambda_1 = 3$, and $\lambda_2=8$.
This guarantees estimability of all parameters under consideration as the range inclusion condition \eqref{eq:cond:range:inclusion} for the information matrix $M(\mu_{\text{rounded}})$ is satisfied (that is, we put $q=2$ and $k_i=i$ for $i=0,1,2$ in the notation of Section~\ref{sec:opt2}).
Note that the $\Phi_{E_1}$-criterion coincides with the $E$-optimality criterion.
While the design $\mu_{\text{rounded}}$ is rather inefficient for small $s$ its $\Phi_{E_s}$-efficiency is increasing with respect to $s$ and is close to $1$ if $s=55$.
The $E$-optimality criterion arises also as a special case of Kiefer's $\Phi_p$-criteria  for $p=-\infty$.The values $p=-1$ and $p=0$, corresponding to the very popular  $A$- and $D$-optimality criterion for which we have
\begin{equation*}
    \mathrm{eff}_{\Phi_{-1}} (\mu_{\text{rounded}}) \approx 2.59~\% % 0.025933073875556607
\end{equation*}
and
\begin{equation*}
    \mathrm{eff}_{\Phi_{0}} (\mu_{\text{rounded}}) \approx 33.93~\%\,.% 0.33931908403906247
\end{equation*}
Thus the $\lambda_8$-designs obtained by our approach leads  a substantial improvement of the designs, which are currently available.

\begin{figure}[h]
     \centering
\includegraphics[height=6.5cm,width=6.5cm]{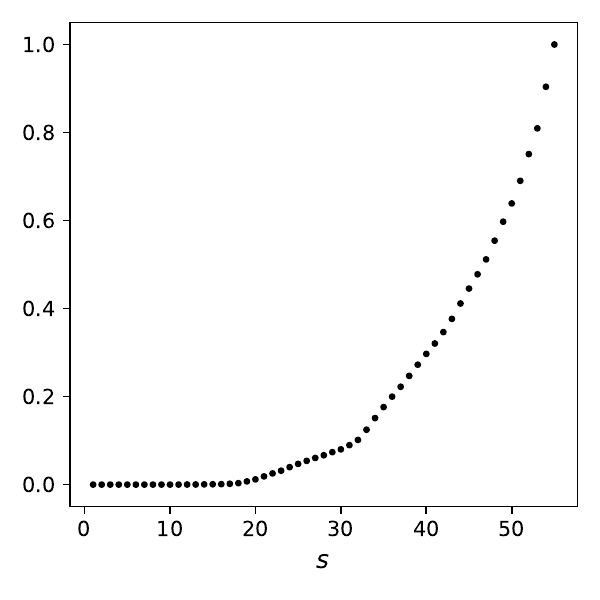}
     \includegraphics[height=6.5cm,width=6.5cm]{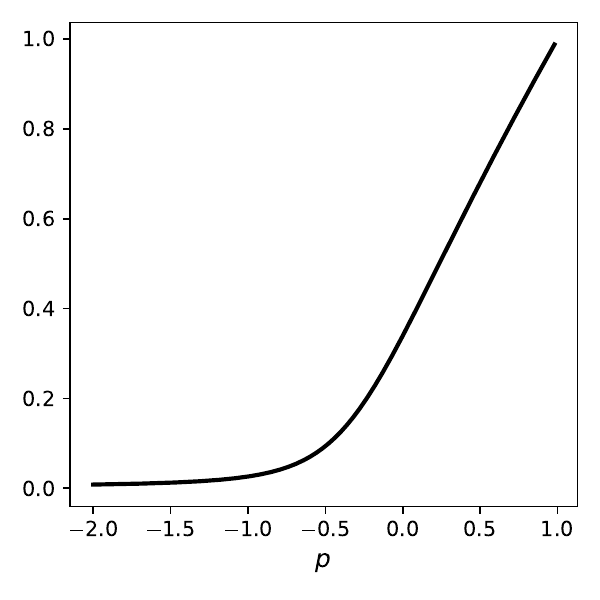}
    \caption{\it $\Phi_{E_s}$-efficiency (left part) and $\Phi_{p}$-efficiency (right part) of the 
    design $\mu_{\text{rounded}}$ with $n=97$ points. This design is  obtained by applying the rounding procedure of \citet{Pukelsheim1992} to the approximate design with $225$ points and unequal weights determined in Example~3.1 of \citet{Dette2019}. For this sample size a $\lambda_8$-design exists and is both $\Phi_{E_s}$- and $\Phi_p$-optimal (this design is determined numerically and has $n=97$ points.)}
    \label{fig:efficiency:S3}
\end{figure}

\subsubsection{A general procedure for the construction of $t$-designs}
In Section~\ref{subsubsec:SU2:explicit}, we have given an explicit constructions of a $t$-designs on $\mathbb S^3$ for $t=2$.
Similar explicit construction for small values of $t$ can be derived from the literature on $t$-design; see \citet{Bajnok1992}, and the references therein.
However, since the construction of tight $t$-designs is restricted to few potential values of $t$, it is important to have a general procedure that allows to generate $t$-designs of manageable size for any choice of $t$.
A procedure for this purpose has been proposed in the work of ~\citet{Bajnok1992} which we describe in the following.
Given a Lebesgue-integrable function $w(x)$ with positive integral on every non-degenerate subinterval $I$ of $[-1,1]$, a finite set $c_1,\ldots,c_n \in [-1,1]$ of distinct points is called an interval $t$-design for the weight function $w(x)$ if
\begin{equation*}
    \frac{1}{n} \sum_{j=1}^n f(c_j) = \frac{1}{\int_{-1}^1 w(x) \dd x} \int_{-1}^1 f(x)w(x)\dd x
\end{equation*}
for every polynomial $f(x)$ of degree at most $t$.
\citet[Theorem~3]{Bajnok1992} (more precisely, the special case resulting from taking $d=4$ and $m=3$ in the statement of this theorem) states that the set of points $(x_1,x_2,x_3,x_4)$ with
\begin{align*}
    x_1 &= c_{1,j_1},\\
    x_2 &= c_{2,j_2}\sqrt{1-c_{1,j_1}^2},\\
    x_3 &= \cos \left( \frac{2k\pi}{3} \right)  \sqrt{ (1-c_{1,j_1}^2) (1-c_{2,j_2}^2)},\\
    x_4 &= \sin \left( \frac{2k\pi}{3} \right)  \sqrt{ (1-c_{1,j_1}^2) (1-c_{2,j_2}^2)}\,, 
\end{align*}
where $i=1,2$, $k=1,2,3$, $j_i = 1,\ldots,n_i$ and the set $\{ c_{i,1}, c_{i,2}, \ldots, c_{i,n_i} \}$ is an arbitrary weighted interval $t$-design for the weight function $w_i(x) = (1-x^2)^{(2-i)/2}$, forms a spherical $2$-design on $\mathbb S^3$.
The design constructed this way has $3n_1n_2$ distinct points.
Work of \citet{Bajnok1992} shows the lower bounds $n_1 \geq c_1 t^9$ and $n_2 \geq c_2 t^{4.5}$ (the numerical constants $c_1$ and $c_2$ can be made precise).

\subsection{Optimal design on $\SO(3)$}\label{subsec:example:SO3}
We now consider the problem of linear regression on the special orthogonal group $\SO(3)$.
Such problems arise in a wide variety of applications: see, for example, 
\citet{Chirikjian2021,Schaeben2007,Hielscher2013,Kovacs2003}.
Recall that the matrix group $\SO(3)$ is defined as
\begin{equation*}
    \SO(3) = \{ A \in \R^{3 \times 3} \mid A^\top A = A A^\top = I_3 \text{ and } \det(A)=1 \}\,.
\end{equation*}
The natural basis for functions on $\SO(3)$ is given by the
Wigner $D$-functions $D^\ell_{m,n}$ which are for $\ell \in \N_0$  and $m, m' \in \Z$ with $\lvert m \rvert, \lvert m' \rvert \leq \ell$ defined by
\begin{equation*}
    D^\ell_{m,m'} (\alpha,\beta,\gamma) = \ee^{-\ii m \alpha} d^\ell_{m,m'}(\beta) \ee^{-\ii m' \gamma}
\end{equation*}
for Euler angles $\alpha,\gamma \in [0,2\pi)$ and $\beta \in [0,\pi]$, where the Wigner $d$-functions $d^\ell_{m,m'}(\beta)$ are defined by (see \citet{Shen2017} for more details)
\begin{equation*}
    d^\ell_{m,m'}(\beta) = (-1)^\nu \binom{2\ell-k}{k+a}^{1/2} \binom{k+b}{b}^{-1/2} \left( \sin \left( \frac{\beta}{2} \right)\right)^a \left( \cos\left( \frac{\beta}{2} \right) \right)^b P_k^{(a,b)}(\cos \beta)
\end{equation*}
for $k = \ell - \max \{ \lvert m \rvert, \lvert m' \rvert \}$, $a=\lvert m - m'\rvert$, $b= \lvert m + m' \rvert$,
\begin{equation*}
    \nu = \begin{cases}
        0, & \text{ 
       {if} } m' \leq m,\\
        m' - m, & \text{ 
        {if} } m' > m,
    \end{cases}
\end{equation*}
and $P_k^{(a,b)} (x) $ denotes the Jacobi polynomial of degree $k$ with parameters $a$ and $b$ (see, for example, \citet{szego75} for more details).

Representing the regression function by eigenfunctions associated with eigenvalues up to a certain resolution is equivalent to considering Wigner $D$-functions up to a certain threshold, say $\ell \leq L$.
Hence, the regression model  can be formulated as
\begin{equation}\label{eq:model:SO3}
    \Eb [Y|g] = \sum_{\ell=0}^L \sum_{m=-\ell}^\ell \sum_{m'=-\ell}^\ell c^\ell_{m,m'} D^\ell_{m,m'}(g)
\end{equation}
and has
$$
\sum_{\ell=0}^L (2\ell+1)^2 = \frac{(2L+1)(2L+2)(2L+3)}{6}
$$ parameters.
In order to construct an optimal design it is helpful to understand the relationship between the irreducible representations of the groups $\SU(2)$ and $\SO(3)$.
Assuming that the irreducible representations of both groups are parameterized by $\N_0$ (with the irreducible representations ordered with respect to the associated increasing sequence of eigenvalues of the negative Laplace-Beltrami operator), there exists a one-to-one correspondance between the irreducible representations of $\SO(3)$ and the irreducible representations of $\SU(2)$ with even index (see \citet[Proposition~5.6.9]{Kowalski2014} and \citet[Proposition~G.III.15]{Berger1971} for details).
This correspondance can be made explicit via the covering map $p \colon \SU(2) \to \SO(3)$ given by
\begin{equation*}
    p \left( \begin{pmatrix}
        x+\ii y & u + \ii  v\\
        -u + \ii v & x - \ii y
    \end{pmatrix}
    \right) = \begin{pmatrix}
        x^2+ y^2 - u^2 + v^2 & 2(yv+xu) & -2(xv-yu)\\
        -2(xu-yv) & x^2 - y^2 - u^2 + v^2 & 2(xy+uv)\\
        2(xv+yu) & -2(xy-uv) & x^2 - y^2+u^2 -v^2
    \end{pmatrix}\,,
\end{equation*}
and we refer the reader to \citet[p.~246]{Kowalski2014} for the details.
Consequently, if the series expansion in \eqref{eq:model:SO3} is truncated at $L$, the corresponding series for the regression model with predictors on $\SU(2)$, which is obtained by composition with $p$, is truncated at level $2L$.
By Clebsch-Gordan theory, it is hence sufficient construct a $\lambda$-design on $\SU(2)$ for $\lambda=\lambda_{4L}$, which is $\Phi_p$- and $\Phi_{E_s}$-optimal for the corresponding regression problem on $\SU(2)$.
Then, this optimal design can be mapped to $\SO(3)$ by means of the map $p$.\\

For instance, if we consider model \eqref{eq:model:SO3} for $L=1$, one can construct a $\lambda_4$-design on $\SU(2)$ as in Section~\ref{subsubsec:another} and project it to the rotation group via the map $p$.
\citep{Womersley2017} provides two possible solutions for such a $\lambda_4$ design on $\SU(2)$.
The first one is a numerical $\lambda_4$-design consisting of $n=20$ points (file \texttt{sdf004.00020} from \citep{Womersley2017}) which are mapped via the map $p$ to $20$ distinct points of $\SO(3)$.
The preferable solution here, however, is to start with a known regular and exact $\lambda_5$-design on $\SU(2)$ consisting of $n=24$ points (file \texttt{sdr005.00024} from \citep{Womersley2017}).
Since this latter design consists of pairs of antipodal points, its image under the map $p$ contains only $12$ distinct points which form a smaller optimal design for the regression problem \eqref{eq:model:SO3}.
This finding is consistent with the discussion in Chapter~4 of \citet{Graef2009} where $t$-designs on the rotation group $\SO(3)$ are discussed.
The authors of that paper also construct a set of rotations of cardinality equal to $12$ (denoted by $\mathcal X_T$ in that paper and nicely illustrated in the left plot of their Figure~2).
It is shown that the elements of this set can be used as the nodes of a quadrature formula on $\SO(3)$ with equal weights that reproduces integrals of polynomials of degree $\leq 2$ exactly.
This property is equivalent with the one defining a $\lambda_2$-design on the rotation group.
\citet{Graef2009} conjecture that $12$ is indeed the minimal number of nodes for such a quadrature formula.
Quadrature rules with non-equal weights on the rotation group have also been studied in \citet{Khalid2016}.

At this point, we can already compare the  optimal design with $12$ points 
constructed by our new method with designs obtained by an equally spaced sampling of Euler angles. For this purpose we denote by  $\mu_{(n_\alpha,n_\beta,n_\gamma)} $  the design with $n_\alpha,n_\beta$ and $n_\gamma$  equally spaced points on the range of the  three  Euler angles $\alpha  \in [0,2\pi)$, $\beta \in [0,\pi]$ and  $\gamma \in [0,2\pi)$, respectively.
Since the total number of design points is equal to the product $n_\alpha n_\beta  n_\gamma$, one obtains for the total sample size $n=12$  the combinations $(2,2,3)$, $(2,3,2)$, and $(3,2,2)$.
It  turns out that these uniform designs do not allow to estimate all parameters in model \eqref{eq:model:SO3} with $L=1$ as  the corresponding $10 \times 10$
matrices   $M(\mu_{(n_\alpha,n_\beta,n_\gamma)})$ do not have full rank. In contrast, our approach  produces a design with $12$ points and information matrix $I_{10}$, such that all parameters are estimable.

Next we investigate designs based on equally spaced Euler angles of larger sample size such that the corresponding information matrix has full rank.
More precisely, we consider $n_\alpha,n_\beta$, and $n_\gamma$ 
equally spaced points on the range of the
three Euler angles, where $(n_\alpha,n_\beta,n_\gamma)$ is chosen as $(6,4,6)$, $(8,5,8)$, and $(10,6,10)$.
Therefore the total sample size of design which can be implemented without rounding
is $n=144$, $320$ and $600$, respectively. 
For the  designs $\mu_{(6,4,6)}$, $\mu_{(8,5,8)}$, and $\mu_{(10,6,10)}$ the $\Phi_{p}$-efficiencies are displayed in Figure~\ref{fig:efficiency_comparison_Phip_SO3_Euler} for $-10 \leq p < 1$.
The results impressively show that the same efficiency as equally spaced Euler angle sampling can be obtained with a much smaller sample size.
For example, consider the design $\mu_{(6,4,6)}$, which can be used for a total sample size $n=144$.
The $\Phi_{-\infty}$- and $\Phi_{-1}$-efficiency of this design is given by $0.667$ and $0.907$, respectively.
Therefore using the design with only $8$ observations at each of the $12$  points constructed by the method proposed in this paper yields the same estimation accuracy with respect to the $\Phi_{-\infty}$-criterion as the 
equally spaced Euler angle sampling with $n=144$ observations. Similarly,  the  design with $n=132$ yields approximately the same $\Phi_{-1}$-efficiency.

\begin{figure}[h]
\includegraphics[width=0.7\textwidth]{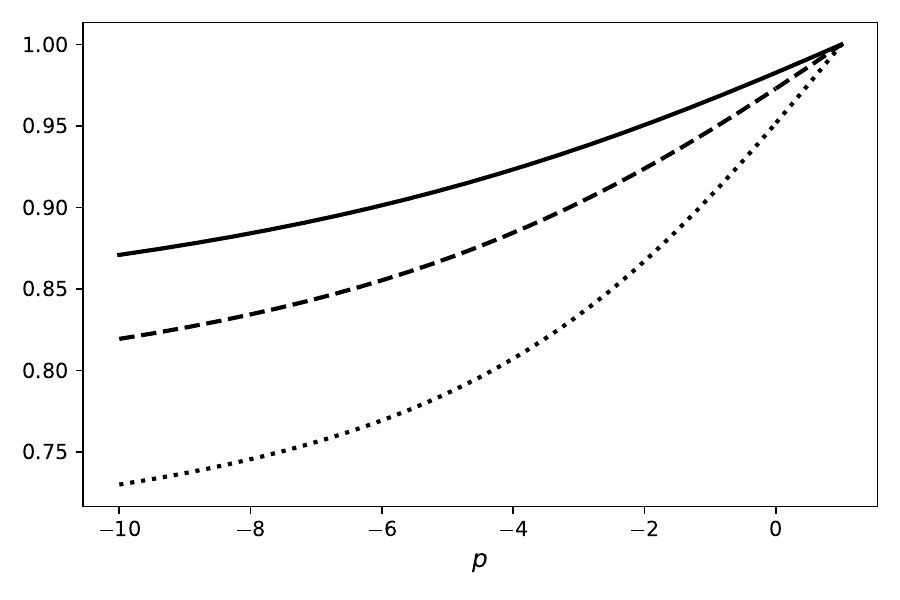}
\caption{\it $\Phi_{p}$-efficiencies of the designs $\mu_{(6,4,6)}$ (dotted line), $\mu_{(8,5,8)}$ (dashed line), and $\mu_{(10,6,10)}$ (solid line). These designs are based on uniform designs for the three Euler angles and have equal weights  at the support  points.  The minimal sample size to obtain an implementable optimal design from the approximate
design is therefore  $144$ $(\mu_{(6,4,6)})$, $320$ $(\mu_{(8,5,8)})$, and $600$ $(\mu_{(10,6,10)})$,  respectively.
For the given regression model \eqref{eq:model:SO3} with $L=1$ there exists a design  with equal weights  at $12$ which is   $\Phi_{E_s}$- and $\Phi_p$-optimal.}
\label{fig:efficiency_comparison_Phip_SO3_Euler}
\end{figure}

\subsection{Optimal design on $\mathbb S^2 \times \SO(3)$}
We conclude  our discussion by taking a closer look at the product manifold $\mathbb S^2 \times \SO(3)$ which is relevant for the application from biology already discussed in the introduction.
For this example, the most convenient basis is given by the  functions
\begin{equation}\label{eq:eigenfunctions:padhorny}
    \varphi_{\ell_{\mathbb S^2},\ell_{\SO(3)},m,m_1,m_2}(\theta,\phi,\alpha,\beta,\gamma) \defeq Y^{\ell_{\mathbb S^2}}_m(\theta,\phi) D^{\ell_{\SO(3)}}_{m_1,m_2}(\alpha,\beta,\gamma)~,
\end{equation}
where $0 \leq \ell_{\mathbb S^2} \leq L_{\mathbb S^2}$ and $0 \leq \ell_{\SO(3)} \leq L_{\SO(3)}$ are the indices corresponding to ${\mathbb S^2}$ and  $ {\SO(3)}$, respectively.
The work  of  
\citet{Padhorny2016} uses slightly different basis functions (see Equation~~[S1] in the Supporting Information in \citet{Padhorny2016}).
In~\eqref{eq:eigenfunctions:padhorny}, the second factor $D^{\ell_{\SO(3)}}_{m_1,m_2}(\alpha,\beta,\gamma)$ and the Euler angles $\alpha$, $\beta$, and $\gamma$ are defined exactly as in Section~\ref{subsec:example:SO3}.
The  functions $Y^{\ell_{\mathbb S^2}}_m(\theta,\phi)$ are the 
spherical harmonics which form an orthogonal basis of
$L^2(\mathbb S^2)$.
The spherical harmonics are usually defined in terms of spherical coordinates where $\theta \in [0,\pi]$ denotes the polar angle and $\varphi \in [0,2\pi)$ the azimuthal angle.
The eigenvalues associated with the eigenfunctions in~\eqref{eq:eigenfunctions:padhorny} are $\ell_{\mathbb S^2} (\ell_{\mathbb S^2}+1)\ell_{\SO(3)}(\ell_{\SO(3)}+1)$.
Again, by Clebsch-Gordan theory, the product of two eigenfunctions of the form \eqref{eq:eigenfunctions:padhorny} with $\ell_{\mathbb S^2} \leq L_{\mathbb S^2}$ and $\ell_{\SO(3)} \leq L_{\SO(3)}$ can be written as the linear combination of eigenfunctions of the same form with indices $\ell_{\mathbb S^2} \leq 2L_{\mathbb S^2}$ and $\ell_{\SO(3)} \leq 2L_{\SO(3)}$.
Using the basis functions from \eqref{eq:eigenfunctions:padhorny}, we consider the regression model given by
\begin{equation}\label{eq:model:padhorny}
    \Eb [Y|g] = \sum_{\ell_{\mathbb S^2}=0}^{L_{\mathbb S^2}} \sum_{\ell_{\SO(3)}=0}^{L_{\SO(3)}} \sum_{m=-\ell_{\mathbb S^2}}^{\ell_{\mathbb S^2}} \sum_{m_1,m_2=-\ell_{\SO(3)}}^{\ell_{\SO(3)}} c_{\ell_{\mathbb S^2},\ell_{\SO(3)},m,m_1,m_2} \varphi_{\ell_{\mathbb S^2},\ell_{\SO(3)},m,m_1,m_2}(g)
\end{equation}
where $g=(\theta,\phi,\alpha,\beta,\gamma)$.
It follows that a $\Phi_p$-optimal design for model \eqref{eq:model:padhorny} can be implemented as a product design of a spherical $t$-design with $t=2L_{\mathbb S^2}$ (covering the $\mathbb S^2$ part) and a $\lambda$-design with $\lambda = \lambda_{2L_{\SO(3)}}$ for the $\SO(3)$ part (which is again obtained by projecting a $\lambda_{4L_{\SO(3)}}$ design on $\SU(2)$ via the map $p$ as  defined in Section~\ref{subsec:example:SO3}).\\

\begin{figure}%[b]
\includegraphics[width=0.7\textwidth]{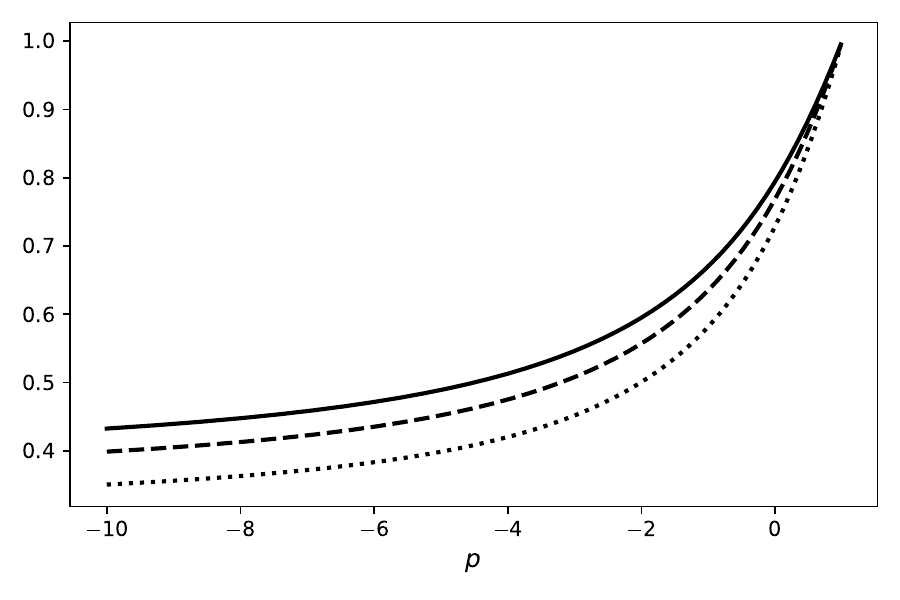}
\caption{\it $\Phi_{p}$-efficiencies of the designs $\mu_{(4,6,6,4,6)}$ (dotted line), $\mu_{(5,8,8,5,8)}$ (dashed line), and $\mu_{(6,10,10,6,10)}$ (solid line). These designs are based on an equidistant sampling of the two spherical coordinates and the three Euler angles and have equal weights of the design points.
The minimal sample size to obtain 
an implementable optimal design  from the approximate
design is therefore $3456$ ($\mu_{(4,6,6,4,6)}$), $12800$ ($\mu_{(5,8,8,5,8)}$)
and $36000$ ($\mu_{(6,10,10,6,10)}$),  respectively. 
For the  given  regression model \eqref{eq:model:padhorny} with $L_{\mathbb S^2}=2$ and $L_{\SO(3)}=1$, there exists a design with equal weights  at $168$ points  that is  $\Phi_{E_s}$- and $\Phi_p$-optimal.}
\label{fig:efficiency_comparison_Phip_Padhorny_Euler_I}
\end{figure}

To be concrete, let us consider the case $L_{\mathbb S^2}=2$ and $L_{\SO(3)}=1$ resulting in a regression model \eqref{eq:model:padhorny} with $90$ parameters. For the Euler angles $(\alpha,\beta,\gamma)$ we use again the design consisting of $12$ points that was constructed in Section~\ref{subsec:example:SO3}.
For the spherical $4$-design on the sphere $\mathbb S^2$ we again resort to a precomputed design consisting of $14$ points (file \texttt{sf004.00014} from \citep{Womersley2015}).

Hence, the resulting optimal product design has $12 \cdot 14 = 168$  points.
In Figure~\ref{fig:efficiency_comparison_Phip_Padhorny_Euler_I} the efficiency curves of three designs based on an equally spaced sampling of the angles $\theta$, $\phi$, $\alpha$, $\beta$, and $\gamma$ are shown.
Here, we again denote the number of sampling points for the respective angles by $n_\theta$, $n_\phi$, $n_\alpha$, $n_\beta$, and $n_\gamma$.

For the vector $(n_\theta,n_\phi,n_\alpha,n_\beta,n_\gamma)$, we consider the three (approximate) designs $\mu_{(4,6,6,4,6)}$, $\mu_{(5,8,8,5,8)}$ and $\mu_{(6,10,10,6,10)}$, and one requires corresponding total sample size $3456$, $12800$ and $36000$ to implement these designs without a rounding procedure.
We again demonstrate the advantage of using the method proposed in this paper by comparison with the uniform design $\mu_{(4,6,6,4,6)}$, which has $\Phi_{-1}$-efficiency $0.582$ and $\Phi_{-\infty}$-efficiency $0.295$. 
This design requires $3456$ observations for its implementation.
In this case, a design taking $12$ observations at each of the $168$ points selected by our method already yields the same estimation accuracy   with a total sample size $n=2016$
(measured with respect to the $\Phi_{-1}$-criterion). For the $\Phi_{-\infty}$-criterion
the difference is even more substantial. Here about $6$ observations at each of the $168$ points selected by our method are sufficient resulting in a design with total sample size $n=1008$, which is a substantial improvement compared to  the $3456$ point design obtained from $\mu_{(4,6,6,4,6)}$.

\begin{acks}[Acknowledgments]
%The work of Holger Dette and Martin Kroll was 
The authors were partially supported by the Deutsche  Forschungsgemeinschaft (DFG) under grant  number 451920280. 
\end{acks}

\bibliographystyle{imsart-nameyear} % Style BST file (imsart-number.bst or imsart-nameyear.bst)
\bibliography{Literature.bib}       % Bibliography file (usually '*.bib')

\newpage

\appendix

\section{Lie groups}\label{app:lie}

\subsection{Basic notions}

An $n$-dimensional \emph{topological manifold} is a second countable Hausdorff topological space $M$ that is locally homeomorphic to an open subset of $\R^n$.
An $n$-dimensional \emph{smooth manifold} is a topological manifold $M$ along with a maximal collection of charts, $\{ \varphi\colon U_\alpha \to V_\alpha \}$, such that
\begin{enumerate}
    \item[(1)] $M = \bigcup_{\alpha} U_\alpha$,
    \item[(2)] for every $\alpha,\beta$ with $U_\alpha \cap U_\beta \neq \emptyset$, the \emph{transition map} 
    \begin{equation*}
        \varphi_{\alpha,\beta} = \varphi_\beta \circ \varphi_\alpha^{-1}\colon \varphi_\alpha(U_\alpha \cap U_\beta) \to \varphi_\beta(U_\alpha \cap U_\beta)
    \end{equation*} is a smooth map on $\mathbb R^n$. Here, the notion \emph{smooth} should be understood as a synonym for $\mathcal C^\infty$.
\end{enumerate}

A \emph{Lie group} is a set $G$ which is both a group with respect to an operation $G \times G \to G, (g,h) \mapsto g\cdot h$, and a smooth manifold such that the algebraic structure and the manifold structure are compatible in the sense that both the multiplication map $G \times G \to G, (g,h) \mapsto gh$ and the inversion map $G \to G, g \mapsto g^{-1}$ are smooth.
Real Lie groups are orientable \citep[Theorem~4.34, (1)]{Kirillov2008}.
The \emph{Lie algebra} $\gf$ associated with a Lie group $G$ is defined as the tangent space $T_e G$ to $G$ at the identity element $e \in G$.
The product structure on $\gf$ is given by the %so-called 
\emph{Lie bracket} or \emph{commutator} on $\gf \times \gf$ which is a bilinear map $\gf \times \gf \to \gf, (X,Y) \mapsto [X,Y] \in \gf$ that is (1) skew-symmetric (that is, $[X,Y]=-[Y,X]$ for all $X,Y\in\mathfrak g$) and (2) satisfies the Jacobi identity
\begin{equation*}
[X,[Y,Z]] + [Y,[Z,X]] + [Z,[X,Y]] = 0  
\end{equation*} for all $X,Y,Z\in\mathfrak g$.
For any $g \in G$, denote with $C_g$ the map defined by conjugation with $g$,
\begin{equation*}
    C_g\colon G \to G, h \mapsto ghg^{-1}.
\end{equation*}
We denote with $\Ad_g \defeq T_e C_g\colon \gf \to \gf$ the corresponding tangent map given by
\begin{equation*}
    \Ad_g(X) = (C_g)_\ast X.
\end{equation*}
For any $g,h \in G$, $\Ad_g \circ \Ad_h = \Ad_{gh}$ holds, and therefore the map $\Ad \colon G \to \GL(\gf), g \mapsto \Ad_g$ defines a representation of  $G$ (called the \emph{adjoint representation}).
Now take the tangent map of $\Ad$ at the identity element $e \in G$ to obtain the map
\begin{equation*}
    \ad \defeq T_e \Ad \colon \gf \to \End(\gf)
\end{equation*}
where $\End(\gf)$ denotes the set of endomorphisms on $\gf$.
In particular, for any $X$ in $\gf$, one has a linear map $\ad\colon \gf \to \gf$ which is actually given by $\ad_X(Y)=[X,Y]$.
Now, the \emph{Killing form} $\kappa$ is a symmetric bilinear form on $\gf$ defined by
\begin{align*}
    \kappa\colon \gf \times \gf \to \R, \quad (X,Y) \mapsto \tr(\ad_X \circ \ad_Y).
\end{align*}
As usual, $\tr(\phi)$ is defined as the sum of diagonal entries with respect to a matrix representation of the endomorphism $\phi \in \End(\gf)$ (of course, this definition does not depend on the choice of a specific basis).
In general, the Killing form is negative semidefinite.
The %so-called 
\emph{semi-simple} Lie algebras are exactly the ones with non-degenerate Killing form (sometimes the non-degeneracy of the Killing form is even used to define the notion of semi-simplicity).
A Lie group is called \emph{semi-simple} if its Lie algebra is semi-simple.
The Lie algebra $\kf$ of a \emph{compact} Lie group $G$, however, is not necessarily semi-simple.
Nevertheless, the complexification $\gf \defeq \kf_\C$ of $\kf$ is reductive meaning that $\gf$ can be written as the direct sum of a semi-simple algebra and a commutative algebra.

A metric on a Lie group $G$ is called bi-invariant if the mappings $L_h\colon G \to G, x \mapsto hx$ and $R_h\colon G \to G, x \mapsto xh$ are isometries for this metric.
It can be shown that any compact Lie group admits a bi-invariant metric (in general, there is a 1:1-correspondence between bi-invariant metrics on $G$ and certain bilinear forms on the Lie algebra $\gf$).
For a semisimple Lie group, %the 
existence of a bi-invariant metric %and 
is ensured by the negative definiteness of the Killing form.% are In fact, any metric can be decomposed as a direct sum of multiples of the Killing forms on the simple components.

A \emph{representation} of a Lie group $G$ on a topological vector space $V$ is a pair $(\pi,V)$, where $\pi\colon G \to \GL(V)$ is a group homomorphism and the map $G \times V \to V,~ (g,v) \mapsto \pi(g)v$ is smooth.  

A \emph{representation} of a Lie algebra $\gf$ is a pair $(\pi,V)$, where $V$ is a finite dimensional vector space and $\pi\colon \gf \to \End(\gf)$ is a Lie algebra homomorphism.
Note that no extra continuity assumption is needed here because Lie algebras, being tangent space to finite dimensional smooth manifolds, are finite dimensional vector spaces, and linear maps between such vector spaces 
are automatically (bounded hence) continuous.

Although our interest is in real Lie groups and their corresponding Lie algebras, we will consider complex representations.
Since some of the results we refer to are stated for complex representations of complex Lie algebras only, we emphasize that by Lemma~4.4 in~\citep{Kirillov2008} the categories of \emph{complex} representations of $\gf$ and its complexification $\gf_\C$ are equivalent and results for both categories are in one-to-one correspondence.
Moreover, the complexification of a semi-simple Lie algebra is again semi-simple.

\subsection{Some details for the proof of Theorem~\ref{thm:lambda:Phi}}\label{app:proof:details}

The proof of our main result in Theorem~\ref{thm:lambda:Phi} 
relies on the %so-called 
theorem of the highest weight for connected and compact (matrix) Lie groups (see \citet[Chapter 12]{Hall2015Lie} for details).
For a connected and compact Lie group $G$ (with fixed maximal torus $T$) the irreducible representations of $G$ are in one-to-one correspondence with the \emph{highest weights}  which are elements of the
corresponding Lie algebra $\tf$ of $T$.
Essential for the proof of Theorem~\ref{thm:lambda:Phi} is a relation between the eigenvalues of the Laplace operator on $G$ and the highest weights of the irreducible representations.
More precisely, it can be shown that the eigenspaces of the Laplace-Beltrami operator on a compact Lie group are in one-to-one correspondance with the (finite-dimensional) irreducible representations of $G$ %\citep
\citep[Proposition~10.2]{Bump2013}, and this equivalence can be made precise via the %so-called 
\emph{Casimir element} of the universal enveloping algebra.

The highest weight $\mu$ associated with a representation corresponds to the eigenvalue $\lambda_\mu = \lVert \mu + \rho \rVert^2 - \lVert \rho \rVert^2$ where $\rho$ denotes half the sum of the positive roots (in this context, roots are to be understood with respect to a fixed maximal torus $T$ of the Lie group resp. its corresponding Lie algebra $\tf$).
The norm $\lVert \cdot \rVert$ used here is the %so-called 
\emph{Killing norm}, that is, the norm induced by the Killing form.

For the proof of Theorem~\ref{thm:lambda:Phi} it is important to understand the behaviour of weights under tensor products.
Let $V_\mu$ resp. $V_\nu$ be two irreducible representations of $G$ with highest weights $\mu$ resp. $\nu$.
Then, the tensor product $V_\mu \otimes V_\nu$ decomposes as the direct sum of finitely many irreducible representations,
\begin{equation*}
    V_\mu \otimes V_\nu = V_{\omega_1} \oplus \ldots \oplus V_{\omega_p},
\end{equation*}
and it can be shown that $\omega_1,\ldots,\omega_p \leq \mu + \nu$.
In particular,
\begin{equation}\label{eq:weights:tensor_product}
    V_\mu \otimes V_\nu \subseteq \bigoplus_{\omega \leq \mu + \nu} V_\omega.
\end{equation}
Restating these results in terms of eigenvalues yields that the eigenvalue $\lambda_{\mu + \nu}$ associated with the sum $\mu + \nu$ (which is the highest weight appearing on the right-hand side of \eqref{eq:weights:tensor_product}) satisfies
\begin{align*}
    \lambda_{\mu + \nu} &= \lVert \mu + \nu + \rho \rVert^2 - \lVert \rho \rVert^2\\[2mm]
    &= \left \langle \mu + \nu + 2 \rho, \mu + \nu \right \rangle\\[2mm]
    &= \left \langle \mu + 2\rho,\mu \right \rangle + \left \langle \nu + 2 \rho, \nu  \right \rangle + \left \langle \mu,\nu \right \rangle +\left \langle \nu,\mu \right \rangle\\[2mm]
    &\leq \lambda_\nu + \lambda_\mu + 2 \cdot \lVert \mu \rVert \cdot \lVert \nu \rVert\\[2mm]
    &= \lambda_\nu + \lambda_\mu + 2 \Big( \lVert \rho \rVert + \sqrt{\lambda_\mu + \lVert \rho \rVert^2} \Big ) \Big ( \lVert \rho \rVert + \sqrt{\lambda_\nu + \lVert \rho \rVert^2} \Big ),
\end{align*}
and the right-hand side can be bounded from above by \begin{equation*}
    (2+2(1+\sqrt{2})^2)\lambda = (8+4\sqrt 2)\lambda < 14\lambda\,,
\end{equation*} provided that $\lambda_\mu, \lambda_\nu \leq \lambda$ and $\lambda \geq \lVert \rho \rVert \vee 1$.
Consequently, for two eigenfunctions $\phi_{\nu}$ and $\phi_\mu$ associated with eigenvalues $\lambda_\nu \leq \lambda$ and $\lambda_\mu \leq \lambda$, respectively, the product $\phi_\nu\phi_\mu^\ast$ satisfies
\begin{equation*}
    \phi_{\nu} \phi_\mu^\ast \in \bigoplus_{\kappa \leq 14\lambda} E_\kappa.
\end{equation*}
More precisely, if $\phi_\nu \neq \phi_\mu$, then
\begin{equation}\label{eq:product:nu:neq:mu}
    \phi_{\nu} \phi_\mu^\ast \in \bigoplus_{0 < \kappa \leq 14\lambda} E_\kappa
\end{equation}
by orthogonality of the basis functions and using the fact that for a compact and connected Riemannian manifold the eigenspace associated with the eigenvalue $\lambda_0=0$ is the one-dimensional space containing the constant functions on $G$.
On the other hand, if $\phi_\nu = \phi_\mu$, then the product $\phi_\nu \phi_\nu^\ast = \lvert \phi_\nu \rvert^2$ can be written in the form $\phi_\nu \phi_\nu^\ast = 1 + \phi$ for some
\begin{equation}\label{eq:product:nu:eq:mu}
    \phi \in \bigoplus_{0 < \kappa \leq 14\lambda} E_\kappa.
\end{equation}
\end{document}